\documentclass[a4paper, 12pt, oneside, reqno, notitlepage]{amsart}
\usepackage[margin=1in]{geometry}
\usepackage{amsmath,amssymb,amsthm,graphicx,mathrsfs,bbm,url}
\usepackage{amsthm}
\usepackage{wrapfig}
\usepackage{enumitem}
\usepackage{mathtools}
\usepackage[utf8]{inputenc}
 \usepackage[T1]{fontenc}
\usepackage[usenames,dvipsnames]{color}
\usepackage[colorlinks=true,linkcolor=Red,citecolor=Green]{hyperref}
\usepackage[super]{nth}
\usepackage[open, openlevel=2, depth=3, atend]{bookmark}
\hypersetup{pdfstartview=XYZ}
\usepackage[font=footnotesize]{caption}
\usepackage[all]{xy}
\usepackage{caption}
\usepackage{subcaption}
\usepackage{hyperref}
\usepackage{float}
\usepackage{mathtools}

\usepackage{tikz}
\usetikzlibrary{decorations.markings,decorations.pathreplacing,calligraphy}

\usepackage[normalem]{ulem}

\usepackage{epstopdf}

\theoremstyle{plain}
\newtheorem{theorem}{Theorem}[section]
\newtheorem*{theorem*}{Theorem}

\newtheorem{lemma}[theorem]{Lemma}
\newtheorem{claim}[theorem]{Claim}
\newtheorem{proposition}[theorem]{Proposition}
\newtheorem{corollary}[theorem]{Corollary}

\theoremstyle{definition}
\newtheorem{definition}[theorem]{Definition}

\theoremstyle{remark}

\numberwithin{equation}{section}

\newcommand{\C}{\mathbb{C}}
\newcommand{\R}{\mathbb{R}}

\newcommand{\N}{\mathbb{N}}

\newcommand{\eps}{\varepsilon}

\newcommand{\mc}{\mathcal}

\newcommand{\dd}{\mathrm{d}}

\DeclareMathOperator{\vol}{vol}

\DeclareMathOperator{\Tr}{Tr}

\newcommand{\be}{\begin{equation}}
\newcommand{\ee}{\end{equation}}

\title
[Marked boundary rigidity for surfaces of Anosov type]
{Marked boundary rigidity for surfaces\\ of Anosov type}

\author{Alena Erchenko}

\address{Department of Mathematics, Dartmouth College, Hanover, NH, USA}

\email{alena.erchenko@dartmouth.edu}

\author{Thibault Lefeuvre}

\address{Université de Paris and Sorbonne Université, CNRS, IMJ-PRG, F-75006 Paris, France.}

\email{tlefeuvre@imj-prg.fr}

\begin{document}

\begin{abstract}
Let $\Sigma$ be a smooth compact connected oriented surface with boundary. A metric on $\Sigma$ is said to be of \emph{Anosov type} if it has strictly convex boundary, no conjugate points, and a hyperbolic trapped set. We prove that two metrics of Anosov type with the same marked boundary distance are isometric (via a boundary-preserving isometry isotopic to the identity). As a corollary, we retrieve the boundary distance rigidity result for simple disks of Pestov and Uhlmann \cite{Pestov-Uhlmann-05}. The proof rests on a new transfer principle showing that, in any dimension, the marked length spectrum rigidity conjecture implies the marked boundary distance rigidity conjecture under the existence of a suitable isometric embedding into a closed Anosov manifold. Such an isometric embedding result for open surfaces of Anosov type was proved by the first author with Chen and Gogolev in \cite{Chen-Erchenko-Gogolev-20} while the marked length spectrum rigidity for closed Anosov surfaces was established by the second author with Guillarmou and Paternain in \cite{Guillarmou-Lefeuvre-Paternain-23}.
\end{abstract}

\maketitle

\section{Introduction}

\subsection{Marked boundary distance rigidity}

Let $\Sigma$ be a smooth compact connected oriented Riemannian manifold with boundary. A metric $g$ on $\Sigma$ is said to be of \emph{Anosov type} if the boundary $\partial \Sigma$ is strictly convex, the metric has no conjugate points and the set of trapped geodesics is hyperbolic (see Definition \ref{definition:anosov} for further details). Given $x,y \in \partial \Sigma$, we denote by $\mc{C}_{x,y}$ the set of all homotopy classes of curves with fixed endpoints $x$ and $y$. The Anosov assumption guarantees that for every $x,y \in \partial \Sigma$, for every homotopy class of curves $c \in \mc{C}_{x,y}$, there exists a unique $g$-geodesic $\gamma_{x,y}(c)$ joining $x$ to $y$ (see \cite[Lemma 2.2]{Guillarmou-Mazzucchelli-18} for instance). We set $\mc{P} := \left\{ (x,y,c) ~|~ x,y \in \partial \Sigma, c \in \mc{C}_{x,y} \right\}$. The \emph{marked boundary distance function} is then defined as
\begin{equation}
\label{equation:mbd}
d_g : \mc{P} \to [0,\infty), \qquad d_g(x,y,c) := \ell_{g}(\gamma_{x,y}(c)).
\end{equation}
A folklore conjecture is that the marked boundary distance function should determine the metric up to isometry in this class, see \cite[Conjecture 1.6]{Cekic-Guillarmou-Lefeuvre-22} for instance. In the specific case where $\Sigma$ is diffeomorphic to a ball, this is known as the \emph{boundary rigidity conjecture}, see \cite{Michel-81, Croke-91}. In this article we prove the marked boundary distance conjecture for surfaces:


\begin{theorem}
\label{theorem:main}
Let $\Sigma$ be a smooth compact connected oriented surface with boundary. Let $g_1, g_2$ be two metrics of Anosov type on $\Sigma$. If $g_1$ and $g_2$ have same marked boundary distance function, that is $d_{g_1} = d_{g_2}$, then there exists a smooth diffeomorphism $\phi \in \mathrm{Diff}_0(\Sigma,\partial \Sigma)$ such that $g_1 = \phi^*g_2$.
\end{theorem}

The group $\mathrm{Diff}_0(\Sigma,\partial \Sigma)$ is defined as the set of all diffeomorphisms of $\Sigma$ fixing the boundary, and isotopic to the identity through a path of diffeomorphisms preserving the boundary. The fact that the isometry $\phi$ belongs to $\mathrm{Diff}_0(\Sigma,\partial \Sigma)$ is specific to surfaces; in higher dimensions, one should only expect $\phi$ to fix the boundary, see \S\ref{ssection:topology} where this is further discussed.

Recall that a Riemannian manifold with boundary is said to be \emph{simple} if it has strictly convex boundary, no conjugate points, and no trapped geodesics. This forces the manifold to be diffeomorphic to a ball, see \cite[Chapter 3, Section 3.8]{Paternain-Salo-Uhlmann-book}. In particular, simple manifolds are of Anosov type; in this case, $\mc{P}$ in \eqref{equation:mbd} is simply reduced to $\partial \Sigma \times \partial \Sigma$ (there is a unique geodesic connecting boundary points) and $d_g$ is called the \emph{boundary distance function}. Taking $\Sigma$ to be a disk, Theorem \ref{theorem:main} implies as a corollary the boundary rigidity result of Pestov-Uhlmann \cite{Pestov-Uhlmann-05}: two simple disks with same boundary distance function are isometric (via a boundary-preserving isometry). This will be further discussed in Corollary \ref{corollary:2} below.

The Pestov-Uhlmann rigidity theorem answered a conjecture first raised by Michel \cite{Michel-81} in 1981. For metrics within the same conformal class, Theorem \ref{theorem:main} was obtained by Croke-Dairbekov \cite{Croke-Dairbekov-04}. In the case of negatively-curved surfaces, Theorem~\ref{theorem:main} was proved by Guillarmou-Mazzucchelli \cite{Guillarmou-Mazzucchelli-18}, using ideas of Croke \cite{Croke-90} and Otal \cite{Otal-90,Otal-90-2}.

However, it is clear from their proof that it does \emph{not} extend beyond the setting of negatively-curved metrics. Theorem \ref{theorem:main} therefore rests on new tools: it will be mainly obtained as a corollary of Theorem \ref{theorem:transfer} below showing that the rigidity of the marked length spectrum on \emph{closed Anosov} manifolds implies the rigidity of the marked boundary distance function \eqref{equation:mbd} whenever the metrics considered can be \emph{isometrically embedded} into closed Anosov manifolds. We call this a \emph{transfer principle} from closed to open manifolds, see \S\ref{ssection:transfer} where this proof strategy is further discussed.

We also emphasize that the knowledge of the marked boundary distance function on a manifold with boundary of Anosov type is similar to that of the marked length spectrum on a closed Anosov manifold: in fact, in both cases, the problem of injectivity can be reformulated by asking whether a smooth conjugacy between two geodesic flows (isotopic to the identity, and equal to the identity in the boundary case) comes from an isometry on the base between the two metrics.

In higher dimension and for more general manifolds, the boundary rigidity problem was studied by several authors, see \cite{Gromov-83, Sharafutdinov-94, Stefanov-Uhlmann-98, Stefanov-Uhlmann-05, Burago-Ivanov-10, Lefeuvre-19-1} among other references. More recently, a microlocal approach developed by Stefanov-Uhlmann-Vasy \cite{Uhlmann-Vasy-16, Stefanov-Uhlmann-Vasy-16, Stefanov-Uhlmann-Vasy-17} established boundary rigidity in dimension $\geq 3$ under the extra assumption that the manifold is foliated by strictly convex hypersurfaces. The boundary rigidity problem is also intimately connected to the \emph{lens rigidity} problem which seeks to determine if the lens data, that is the scattering relation for the geodesic flow together with travel times of geodesics (see \S\ref{ssection:scattering} for a definition of the lens data), encodes the Riemannian structure of the manifold. We refer to \cite{Stefanov-Uhlmann-09, Guillarmou-17-2, Cekic-Guillarmou-Lefeuvre-22} for further details.  \\

As a corollary of Theorem \ref{theorem:main}, we will also obtain the following result which is connected to the \emph{minimal filling problem} introduced by Gromov \cite{Gromov-83}:

\begin{corollary}
\label{corollary}
Let $\Sigma$ be a smooth compact connected oriented surface with boundary. Let $g_1, g_2$ be two metrics of Anosov type on $\Sigma$. If $g_1$ and $g_2$ satisfy $d_{g_1} \geq d_{g_2}$, then $\vol_{g_1}(\Sigma) \geq \vol_{g_2}(\Sigma)$ with equality if and only if there exists a smooth diffeomorphism $\phi \in \mathrm{Diff}_0(\Sigma,\partial \Sigma)$ such that $g_1 = \phi^*g_2$.
\end{corollary}

Here, $\vol_g(\Sigma)$ denotes the Riemannian volume of $\Sigma$. The previous corollary follows from \cite[Theorem 1.4 and Corollary 1.5]{Croke-Dairbekov-04} and Theorem \ref{theorem:main} above. Although the statements of \cite[Theorem 1.4 and Corollary 1.5]{Croke-Dairbekov-04} do not formally allow trapped geodesics, their proofs go through as long as the trapped set has measure 0. In particular, it is applicable to the surfaces of Anosov type (see \S\ref{ssection:geometry} where this is further discussed).

\subsection{A transfer principle: from closed to open manifolds}

\label{ssection:transfer}

Recall that on a smooth closed manifold $M$, a metric $g$ is \emph{Anosov} if its geodesic flow on the unit tangent bundle is Anosov. Typical examples are provided by metrics of negative sectional curvature; this was originally observed by Anosov \cite{Anosov-67} himself. However, there exist also many Anosov metrics with areas of positive sectional curvature, see \cite{Eberlein-72, Ruggiero-91, Donnay-Pugh-03}. In order to state the main result of this paragraph, we first need to introduce the following terminology.

\begin{definition}[Extendable metrics]
\label{definition:extendable}
Let $\Sigma$ be a smooth compact connected oriented manifold with boundary, and let $g$ be a metric of Anosov type on $\Sigma$. We say that $g$ is \emph{extendable} if there exists a smooth closed connected oriented Anosov Riemannian manifold $(M,g')$ (of same dimension as $\Sigma$) and an isometric embedding $\iota : (\Sigma,g) \hookrightarrow (M,g')$. If $g_1$ and $g_2$ are two metrics of Anosov type on $\Sigma$, we say that they are \emph{consistently extendable} if they are both extendable to the same manifold $M$, and the extensions $g_1'$ and $g_2'$ coincide on $M \setminus \Sigma$.
\end{definition}

Observe that a necessary condition for $g_1$ and $g_2$ to be consistently extendable is that their $C^\infty$-jets coincide on the boundary $\partial \Sigma$. Given a smooth closed manifold $M$ carrying Anosov metrics, we denote by $\mc{C}$ the set of free homotopy classes of $M$ (this is in $1$-to-$1$ correspondence with conjugacy classes of $\pi_1(M)$). Then, for every Anosov metric $g$, there exists a unique closed geodesic $\gamma_g(c) \in c$ in each free homotopy class $c \in \mc{C}$. One can thus define the \emph{marked length spectrum} map by
\begin{equation}
    \label{equation:mls}
  L_g : \mc{C} \to (0,\infty), \qquad  L_g(c) := \ell_g(\gamma_g(c)).
\end{equation}
In negative curvature, the Burns-Katok Conjecture \cite{Burns-Katok-85} (also known as the marked length spectrum rigidity Conjecture) asserts that the map \eqref{equation:mls} is injective, that is if two metrics $g_1$ and $g_2$ with negative sectional curvature have same marked length spectrum, then there exists a diffeomorphism $\phi : M \to M$, isotopic to the identity, such that $\phi^*g_1 = g_2$. More generally, it is believed that the Burns-Katok Conjecture should still hold in the setting of Anosov metrics. In dimension two, this problem was recently addressed for Anosov surfaces by Guillarmou, the second author and Paternain \cite{Guillarmou-Lefeuvre-Paternain-23}. In negative curvature, the Burns-Katok Conjecture was established on surfaces independently by Croke \cite{Croke-90} and Otal \cite{Otal-90}. In higher dimensions, a partial rigidity result was obtained by Hamenstädt when one of the two metrics is locally symmetric \cite{Hamenstadt-99}, and local rigidity of the marked length spectrum was settled by Guillarmou and the second author \cite{Guillarmou-Lefeuvre-19} (see also \cite{Guillarmou-Knieper-Lefeuvre-19} for an alternative proof with Knieper based on the notion of geodesic stretch).

The marked boundary rigidity problem on open manifolds shares many similarities with the marked length rigidity problem on closed manifolds. For instance, in the case of negatively-curved surfaces, the injectivity of \eqref{equation:mbd} was proved in \cite{Guillarmou-Mazzucchelli-18} by adapting Otal's argument \cite{Otal-90} for the marked length spectrum on closed surfaces. However, as far as Anosov surfaces are concerned, it seems that the arguments of \cite{Guillarmou-Lefeuvre-Paternain-23} do not transfer directly to the open case: the main reason is that, similarly to \cite{Guillarmou-Lefeuvre-Paternain-23}, and using earlier results by Guillarmou \cite[Theorem 4]{Guillarmou-17-2}, it is possible to show that two metrics with the same marked boundary distance function are conformally equivalent on every finite cover via a diffeomorphism preserving the boundary; however, it is not clear that this implies that the metrics are conformally equivalent via a diffeomorphism preserving the boundary and isotopic to the identity. The way we circumvent this difficulty is to show a more general transfer principle, valid in any dimension, relating the injectivity of the marked boundary distance function \eqref{equation:mbd}, and that of the marked length spectrum \eqref{equation:mls} for extendable metrics:

\begin{theorem}[Transfer principle]
\label{theorem:transfer}
Let $\Sigma$ be a smooth compact connected oriented manifold with boundary. Let $g_1, g_2$ be two smooth metrics of Anosov type on $\Sigma$. Assume that $g_1$ and $g_2$ have same marked boundary distance function, that is $d_{g_1} = d_{g_2}$, and that the metrics are consistently extendable to a closed manifold $M$. Further assume that the marked length spectrum is injective on $M$ for Anosov metrics of finite regularity. Then there exists a smooth diffeomorphism $\phi \colon \Sigma \to \Sigma$ such that $\phi|_{\partial \Sigma}=\mathbf{1}_{\partial \Sigma}$ and $g_1 = \phi^*g_2$. 
\end{theorem}

Unlike the $2$-dimensional case, the diffeomorphism $\phi$ is \emph{not} necessarily an element of $\mathrm{Diff}_0(\Sigma,\partial \Sigma)$, see \S\ref{ssection:topology} where this is further discussed. Injectivity of the marked length spectrum for Anosov metrics of finite regularity means that there exists an integer $k_0 \geq 0$ such that if $g_1$ and $g_2$ are $C^{k_0}$-regular Anosov metrics on $M$ such that $L_{g_1}=L_{g_2}$, then they are isometric (via a $C^{k_0+1}$-regular diffeomorphism). The need to consider metrics of finite regularity comes from a technical artifact in the proof (see Proposition \ref{proposition:technical}); however, it is harmless in practice.

Theorem \ref{theorem:main} will then be a consequence of the combination of: Theorem \ref{theorem:transfer} above, the injectivity of the marked length spectrum for Anosov surfaces \cite{Guillarmou-Lefeuvre-Paternain-23} for $C^4$-regular metrics, and the fact that metrics of Anosov type on surfaces with boundary are extendable, which follows from the work by the second author with Chen and Gogolev in \cite{Chen-Erchenko-Gogolev-20}.  \\

A consequence of Theorem \ref{theorem:transfer}, which may be of independent interest, is that the Burns-Katok conjecture \cite{Burns-Katok-85} of marked length spectrum rigidity (for Anosov metrics) implies Michel's conjecture \cite{Michel-81} of boundary rigidity of simple manifolds:

\begin{corollary}
\label{corollary:2}
In dimension $n \geq 2$, marked length spectrum rigidity of Anosov $n$-manifolds implies boundary rigidity of simple $n$-manifolds. 
\end{corollary}

The previous corollary follows from Theorem \ref{theorem:transfer} and the fact that simple manifolds are extendable by \cite{Chen-Erchenko-Gogolev-20}. Indeed, when $\Sigma$ is diffeomorphic to a ball, a metric $g$ on $\Sigma$ is simple if and only if it is of Anosov type; similarly to Proposition \ref{proposition:extendability}, \cite{Chen-Erchenko-Gogolev-20} allows to show that two simple metrics on $\Sigma$ with same boundary distance functions are consistently extendable and this is sufficient to conclude.

Moreover, in the same way, the combination of Theorem \ref{theorem:transfer}, \cite[Theorem 1 and Corollary 1.1]{Guillarmou-Lefeuvre-19} and \cite{Chen-Erchenko-Gogolev-20} partially recovers local marked boundary rigidity for manifolds of Anosov type with non-positive curvature proved by the second author \cite[Theorem 1.1 and Corollary 1.1]{Lefeuvre-19-2}. We note that in \cite{Chen-Erchenko-Gogolev-20} the extensions are constructed only under the additional assumption that the boundary components of the manifold are diffeomorphic to a sphere or to $S^1 \times S^{n-1}$; in a forthcoming paper \cite{GuedesBonthonneau-23}, Guedes Bonthonneau is able to remove this restriction on the topology of the boundary components when $n=3$, thus showing that any $3$-manifold with boundary and of Anosov type is extendable in the sense of Definition \ref{definition:extendable}. \\

\noindent \textbf{Acknowledgement:} The authors wish to thank the CIRM, where part of this article was written, for support and  hospitality. The authors also thank C. Guillarmou, G. Paternain, A. Wilkinson for valuable comments on an earlier version of this manuscript. AE was supported by the Simons Foundation (Grant Number 814268 via the Mathematical Sciences Research Institute, MSRI) and NSF grant DMS-1900411.

\section{Preliminaries}

Throughout this section, $(\Sigma,g)$ is a smooth compact connected oriented $n$-dimensional Riemannian manifold with boundary. 

\subsection{Geometric assumptions}

\label{ssection:geometry}

Let $S\Sigma$ be the unital tangent bundle of $(\Sigma,g)$ and denote by $\pi : S\Sigma \to\Sigma$ the footpoint projection, $(\varphi_t)_{t \in \R}$ the geodesic flow on $S\Sigma$ (possibly incomplete) and $X \in C^\infty(S\Sigma, T(S\Sigma))$ its infinitesimal generator.
The \emph{Liouville $1$-form} $\lambda$ is defined such that $\iota_X \lambda = 1, \iota_X \dd \lambda = 0$; it is invariant by the geodesic flow, namely $\mc{L}_X \lambda = 0$. Moreover, $\dd\lambda$ is non-degenerate and this allows to define the \emph{Liouville volume form}
\[
\mu := \dfrac{1}{(n-1)!} \lambda \wedge \dd\lambda^{\wedge (n-1)},
\]
which is also invariant by the geodesic flow.

We then define
\[
\partial_{\mp}S\Sigma := \left\{ (x,v) \in S\Sigma ~|~ x \in \partial \Sigma, \pm g_x(v,\nu(x)) < 0 \right\},
\]
where $\nu(x)$ denotes the unit outward-pointing normal vector to $T\partial \Sigma$ at $x \in \partial \Sigma$. The set $\partial_-S\Sigma$ (resp. $\partial_+S\Sigma$) corresponds to all inward-pointing (resp. outward-pointing) vectors at the boundary. The trapped tails $\Gamma_\mp$ are defined as
\begin{equation}
\label{equation:splitting}
\Gamma_\mp := \left\{ (x,v) \in S\Sigma ~|~ \forall t \geq 0, \pi(\varphi_{\pm t}(x,v)) \in \Sigma\right\},
\end{equation}
and the trapped set $K := \Gamma_- \cap \Gamma_+ \subset S\Sigma$ is the set of all unit vectors trapped both in the past and in the future. The set $K$ is a closed flow-invariant set in $S\Sigma$. By strict convexity of $\partial\Sigma$, we have $K\subset S\Sigma^\circ$.

The trapped set $K$ is said to be \emph{hyperbolic} if there exists a flow-invariant continuous splitting 
\[
T(S\Sigma)|_{K} = \R X \oplus E^s \oplus E^u,
\]
and uniform constants $C,\lambda > 0$ such that
\begin{equation}
\label{equation:anosov}
\begin{array}{ll}
 |d\varphi_t(w)| \leq C e^{-\lambda t}|w|, & \qquad  \forall t \geq 0, \forall w \in E^s, \\
  |d\varphi_{-t}(w)| \leq C e^{-\lambda t}|w|, & \qquad  \forall t \geq 0, \forall w \in E^u.
\end{array}
\end{equation}
We can now introduce the terminology of metrics of \emph{Anosov type}:

\begin{definition}
\label{definition:anosov}
The metric $g$ on $\Sigma$ is said to be of \emph{Anosov type} if:
\begin{enumerate}[label=(\roman*)]
    \item $\partial \Sigma$ is strictly convex, that is the second fundamental form of $g$ is strictly positive on the boundary $\partial \Sigma$;
    \item The metric $g$ has no conjugate points in $\Sigma$;
    \item The trapped set $K$ is hyperbolic in the sense of \eqref{equation:splitting} and \eqref{equation:anosov}.
\end{enumerate}
\end{definition}

It can be proved that if $g$ is of Anosov type, the trapped tails $\Gamma_\pm$ have zero Liouville measure, that is $\mu(\Gamma_- \cup \Gamma_+) = 0$ (see \cite[Section 2.4]{Guillarmou-17-2}).

\subsection{Scattering relation and equivalence of $C^\infty$-jets}

\label{ssection:scattering}

The following result on the equivalence of $C^\infty$-jets at the boundary for metrics with same marked boundary distance function will be needed:

\begin{lemma}
\label{lemma:jets}
Let $\Sigma$ be a smooth compact manifold with boundary, and let $g_1,g_2$ be two metrics on $\Sigma$ of Anosov type. If $d_{g_1} = d_{g_2}$ on $\partial \Sigma$, then there exists a smooth diffeomorphism $\phi \in \mathrm{Diff}_0(\Sigma,\partial \Sigma)$ such that $\phi^*g_1 - g_2$ vanishes to infinite order on $\partial \Sigma$.
\end{lemma}

 We recall that $\mathrm{Diff}_0(\Sigma,\partial \Sigma)$ is the set of all diffeomorphisms of $\Sigma$ fixing the boundary, and isotopic to the identity through a path of diffeomorphisms preserving the boundary.

\begin{proof}
   By \cite[Lemma 2.3]{Guillarmou-Mazzucchelli-18}, we can find a a smooth diffeomorphism $\phi : \Sigma \to \Sigma$, isotopic to the identity via a path of diffeomorphisms fixing the boundary\footnote{This is not mentioned explicitly in \cite[Lemma 2.3]{Guillarmou-Mazzucchelli-18} but it follows from the proof.}, such that $\phi|_{\partial \Sigma} = \mathbf{1}_{\partial \Sigma}$ and $g_1' := \phi^*g_1 = g_2$ on $\partial \Sigma$. The proof of \cite[Theorem 2.1]{Lassas-Sharafutdinov-Uhlmann-03} then shows that if $g_1' = g_2$ on $\partial \Sigma$, the boundary is strictly convex, and $d_{g'_1} = d_{g_2}$, then the $C^\infty$-jets of $g_1'$ and $g_2$ agree on $\partial \Sigma$.
\end{proof}

Recall that if $g$ is a metric on $\Sigma$, the \emph{lens data} of $g$ is defined as the pair $(S_g,\ell_g)$, where $\ell_g : \partial_- S\Sigma \to [0,\infty]$ is the maximal time of existence of the geodesic generated by $(x,v) \in \partial_-S\Sigma$ in $\Sigma$ and $S_g\colon \partial_- S\Sigma\setminus \Gamma_-\rightarrow \partial_+ S\Sigma$ is given by $S_g(x,v) := \varphi_{\ell_g(x,v)}(x,v)$. The following holds:

\begin{lemma}
\label{lemma:lens}
Let $\Sigma$ be a smooth compact oriented manifold with boundary, and let $g_1,g_2$ be two metrics on $\Sigma$ of Anosov type. If $d_{g_1} = d_{g_2}$ on $\partial \Sigma$ and $g_1 = g_2$ on $T_{\partial \Sigma}\Sigma$, then the lens data of the lifted metrics $\widetilde{g}_1$, $\widetilde{g}_2$ to the universal cover $\widetilde{\Sigma}$ agree, that is $(S_{\widetilde{g}_1},\ell_{\widetilde{g}_1}) = (S_{\widetilde{g}_2},\ell_{\widetilde{g}_2})$.
\end{lemma}

Note that the equality of the two metrics on the boundary allows to compare the scattering and lens map (they live on the same space). Lemma \ref{lemma:lens} is actually an equivalence (i.e. if the lens data agree on the universal cover, then the marked boundary distance are the same), see \cite[Lemma 2.4]{Guillarmou-Mazzucchelli-18} for a proof. Lemma \ref{lemma:lens} can also be reformulated by saying that the lens data agree in every homotopy class, that is: if $x,y \in \partial \Sigma$ and $c \in \mc{C}_{x,y}$, then the unique respective geodesics $\gamma_{x,y}^1, \gamma_{x,y}^2 \in c$ for $g_1$ and $g_2$ satisfy 
\begin{equation}
\label{equation:lens-homotopy}
\gamma_{x,y}^1(0) = \gamma_{x,y}^2(0)=x, \gamma_{x,y}^1(\ell) = \gamma_{x,y}^2(\ell)=y, \qquad \dot{\gamma}_{x,y}^1(0) = \dot{\gamma}_{x,y}^2(0)=v, \dot{\gamma}_{x,y}^1(\ell) = \dot{\gamma}_{x,y}^2(\ell). 
\end{equation}
where $\ell := \ell_{g_1}(x,v) = \ell_{g_2}(x,v) = d_{g_1}(x,y,c) = d_{g_2}(x,y,c)$. Combining both Lemma \ref{lemma:jets} and \ref{lemma:lens}, we get that if $g_1$ and $g_2$ are two metrics of Anosov type on $\Sigma$ such that $d_{g_1} = d_{g_2}$, then there exists a diffeomorphism $\phi\in \mathrm{Diff}_0(\Sigma,\partial \Sigma)$ such that $g_1' := \phi^*g_1$ has same $C^\infty$-jets at $\partial \Sigma$ as $g_2$ and $g_1'$ and $g_2$ have the same lens data in every homotopy class in the sense of \eqref{equation:lens-homotopy}.



\subsection{Functional spaces}

\label{ssection:functional-spaces}

Finally, we end this section by introducing the functional spaces needed in the proof. Let $(M,g)$ be a smooth closed Riemanannian manifold and denote by $\nabla$ the Levi-Civita connection of $g$. For $j \geq 0$, the operator
\[
\nabla^j \colon C^\infty(M) \to C^\infty(M,T^*M^{\otimes j})
\]
defined recursively as
\[
\begin{split}
\nabla^{j+1}f(X_0,...,X_j) & = \nabla_{X_0}(\nabla^jf(X_1,...,X_j)) \\
&- (\nabla^jf(\nabla_{X_0}X_1,...,X_j) + ... + \nabla^jf(X_1,...,\nabla_{X_0}X_j)),
\end{split}
\]
for $f \in C^\infty(M)$, $X_0,...,X_j \in TM$, is a differential operator of order $j$. Note that for $f \in C^\infty(M)$, $x \in M, X \in T_xM$, one has the formula:
\begin{equation}
    \label{equation:formula}
    \nabla^j f(x) (X,...,X) = \partial_t^j f(\gamma(t))|_{t=0},
\end{equation}
where $\gamma(t) := \exp_x^g(tX)$.

Given $x \in M, T \in T^*_xM^{\otimes j}$, we define its norm $|T|_g$ (with respect to $g$) by
\[
|T|_g := \sup_{\substack{X_1,...,X_j \in T_xM, \\ |X_i|=1, \,\forall i}} T(X_1,...,X_j).
\]
The $C^m$-norms are then defined for an integer $m \geq 0$ as 
\[
\|f\|_{C^m} := \sup_{0 \leq j \leq m} \sup_{x \in M} |\nabla^j f(x)|_g. 
\]
Similarly, for $0 < \alpha < 1$, the Hölder norm $C^{m,\alpha}$ is defined as
\[
\|f\|_{C^{m,\alpha}} = \|f\|_{C^m} + \sup_{\substack{x,y \in M, \\ d(x,y) < \iota(g)/10}} \dfrac{|\nabla^m f(x) - \tau_{y \to x} \nabla^m f(y)|_g}{d_g(x,y)^\alpha},
\]
where $\iota(g)$ is the injectivity radius of $(M,g)$ and $\tau_{y \to x} : T_yM \to T_xM$ is the parallel transport (with respect to $\nabla$) along the the unique $g$-geodesic joining $x$ to $y$. 

The space $L^2(M)$ is defined as the completion of $C^\infty(M)$ with respect to the norm
\[
\|f\|_{L^2}^2 := \int_{M} |f(x)|^2~\dd\vol_g(x),
\]
where $\dd\vol_g$ is the Riemannian density induced by the metric $g$. For $m \geq 0$, the Sobolev spaces $H^m(M)$ are then defined by completion of $C^\infty(M)$ with respect to the norm
\[
\|f\|_{H^m} := \sum_{k=0}^m \|\nabla^k f\|_{L^2(M,T^*M^{\otimes k})}.
\]
For $s \geq 0$, the spaces $H^s(M)$ are defined by interpolation of $H^{[s]}$ with $H^{[s]+1}$, where $[s]$ denotes the integer part of $s$. Finally, if $M$ has a boundary $\partial M$, we denote by $H^m_0(M)$ the space of $H^m$-functions on $M$ vanishing on the boundary $\partial M$ (this is well-defined for $m > 1/2$). \\

The norms defined above depend on the background metric $g$; however changing $g$ to $g'$, we obtain equivalent norms and it is therefore useless to keep track of the metric dependence in the norm. We will need the following:

\begin{lemma}
\label{lemma:comparison}
Let $g_0$ be a smooth metric on $M$. There exists $\eps_0 > 0$ such that for all metrics $g$ such that $\|g-g_0\|_{C^m} < \eps_0$, for all $f \in C^\infty(M)$, and for all $0 \leq k \leq m + 1$, for all $x \in M$
\[
1/2 \cdot |\nabla^k_{g}f(x)|_{g} < |\nabla^k_{g_0}f(x)|_{g_0} < 2 \cdot  |\nabla^k_{g}f(x)|_{g}.
\]
\end{lemma}

\begin{proof}
Taking local coordinates $(x_i)_{1 \leq i \leq n}$ around an arbitrary point $p \in M$, and writing $(g_{ij})_{1 \leq i,j \leq n}$ for the metric in these coordinates, one verifies that $\nabla^k_g f(x)$ is a polynomial expression in $\partial^\alpha g_{ij}(x)$ and $\partial^\beta f(x)$ for $|\alpha| \leq k-1$ and $|\beta| \leq k$, where $\alpha,\beta \in \N^n$ and $\partial^\alpha := \partial^{\alpha_1}_{x_1} ... \partial^{\alpha_n}_{x_n}$. The claim then follows.
\end{proof}

\section{Prescription of scalar curvature and application to generic metrics}

\subsection{Local prescription of the scalar curvature} 

Let $(\Sigma,g)$ be a smooth compact Riemannian manifold with boundary (not necessarily connected nor oriented), and denote by $s_g$ its scalar curvature. We shall need some results on the (local) prescription of the scalar curvature on $\Sigma$. This topic is classical in Riemannian geometry and goes back to the pioneering work by Kazdan and Warner \cite{Kazdan-Warner-74, Kazdan-Warner-74-open, Kazdan-Warner-75}. However, since we could not find in the literature the statements we needed (mostly because we work in the setting of manifolds with boundary), we include short self-contained proofs. We introduce the following terminology:

\begin{definition} We say that $(\Sigma,g)$ satisfies the property of \emph{local prescription of the scalar curvature} -- \textbf{(LPSC)} in short -- if the Dirichlet Laplacian $L_g := (n-1)\Delta_g + s_g$ has trivial kernel on $\Sigma$, that is if $f \in H^2_0(\Sigma)$ satisfies $L_g f = 0$, then $f = 0$.
\end{definition}

We use throughout the article the geometers' convention $\Delta_g \leq 0$. Since $s_g$ is real, the operator $L_g$ is selfadjoint on $H^2_0(\Sigma)$ and has discrete spectrum accumulating at $-\infty$; \textbf{(LPSC)} is then equivalent to the fact that $0$ is not an eigenvalue of this operator. It follows from elliptic theory that the \textbf{(LPSC)} property is open with respect to the metric $g$ in the $C^{2,\alpha}$-topology for $0 < \alpha < 1$ as $L_g$ -- hence its eigenvalues -- is then a continuous family of elliptic operators with compact resolvent. The terminology is justified by the following lemma:

\begin{lemma}\label{lemma: LPSC property}
Let $(\Sigma,g)$ be a smooth compact Riemannian manifold with boundary and further assume that $g$ satisfies the \emph{\textbf{(LPSC)}} property. Then for all $m \geq 0$ and $0 < \alpha < 1$, there exists $C,\eps_0 > 0$ such that for all $h \in C^{m,\alpha}(\Sigma)$ with $\|h\|_{C^{m,\alpha}} < \eps_0$, there exists $f \in C_0^{m+2,\alpha}(\Sigma)$ (vanishing at order $0$ on the boundary $\partial \Sigma$) such that $s_{e^{2f}g} = s_g + h$ on $\Sigma$, and $\|f\|_{C^{m+2,\alpha}} \leq C\|h\|_{C^{m,\alpha}}$. Moreover, if $h$ is smooth, then $f$ is also smooth. 
\end{lemma}

The subscript $0$ will always indicate functions vanishing at order $0$ on $\partial \Sigma$. For such a metric $g$, Lemma \ref{lemma: LPSC property} asserts that one can perturb conformally the metric in order to obtain a new metric with prescribed scalar curvature.

\begin{proof}
The proof is based on the implicit function theorem; we first treat the case of finite regularity and then address the case of smooth $h$. The scalar curvature satisfies the following formula with respect to conformal changes (see \cite[Theorem 1.159, f)]{Besse-87}):
\[
s_{e^{2f}g} = e^{-2f}(s_g-2(n-1)\Delta_g f - (n-2)(n-1)|df|^2).
\]
Given $h \in C^{m,\alpha}(\Sigma)$, we are looking for $f \in C^{m+2,\alpha}_0(\Sigma)$ solving $\Phi(f,h)=0$ where
\[
\begin{split}
\Phi(f,h) & := s_{e^{2f}g}-(s_g+h) = e^{-2f}(s_g - (n-1)(2\Delta_g f + (n-2)|df|^2)) - s_g - h.
\end{split}
\]
Observe that the map $\Phi : C^{m+2,\alpha}_0(\Sigma) \times C^{m,\alpha}(\Sigma) \to C^{m,\alpha}(\Sigma)$ is smooth (actually, analytic) in both entries and $\Phi(0,0)=0$. Differentiating with respect to $f$, we obtain:
\begin{equation}
    \label{equation:dphi}
    \partial_f\Phi(f,h)\cdot k = - 2 e^{-2f} (k(s_g-(n-1)(2\Delta_g f + (n-2)|df|^2))+(n-1)(\Delta_g k + (n-2)\langle df, dk\rangle)).
\end{equation}
In particular, at $(f,h)=(0,0)$, we get
\[
\partial_f\Phi(0,0) \cdot k = -2((n-1)\Delta_g+s_g) k.
\]
The Dirichlet Laplacian $L_g = (n-1)\Delta_g+s_g$ is selfadjoint with domain $H^2_0(\Sigma) \subset L^2(\Sigma)$ and Fredholm (by ellipticity of $\Delta_g$), so its Fredholm index is equal to $0$. Moreover, by assumption, the kernel of $L_g$ is assumed to be trivial, so it is an isomorphism $L_g \colon H^2_0(\Sigma) \to L^2(\Sigma)$. It is then standard that $L_g \colon C^{k+2,\alpha}_0(\Sigma) \to C^{k,\alpha}(\Sigma)$ is also a isomorphism for every $k \geq 0$ and $0 < \alpha < 1$, see \cite[pp. 108-109]{Gilbarg-Trudinger-01}. Hence, there exists $\eps_0 > 0$ small enough and a smooth map $\omega\colon C^{m,\alpha}(\Sigma) \cap B_{C^{m,\alpha}}(0,\eps_0) \to C^{m+2,\alpha}_0(\Sigma)$ such that for all $h \in C^{m,\alpha}(\Sigma)$ with $\|h\|_{C^{m,\alpha}} < \eps_0$, $f:=\omega(h)$ is the unique solution (near $f=0$) of $\Phi(f,h)=0$. Since $\omega$ is at least $C^1$, there exists $C > 0$ such that
\[
 \|f\|_{C^{m+2,\alpha}(\Sigma)} = \|\omega(h)\|_{C^{m+2,\alpha}(\Sigma)} = \|\omega(h)-\omega(0)\|_{C^{m+2,\alpha}(\Sigma)}  \leq C\|h\|_{C^m(\Sigma)}.
\]
This proves the claim in finite regularity.

It remains to show that if $h$ is smooth, then $f$ is also smooth. This follows from the Nash-Moser implicit function Theorem, see \cite[Chapter 3, Section C.4]{Alinhac-Gerard-07}. First, we claim that the map $\Phi$ satisfies the tame estimate \cite[Condition (H1), p145]{Alinhac-Gerard-07}. Indeed, we have $\partial^2_f\Phi(f,h)\cdot(k_1,k_2) = e^{-2f} R(f,k_1,k_2)$ with
\[
\begin{split}
R(f,k_1,k_2) &= 4 k_1 k_2 (s_g-(n-1)(2\Delta_g f+(n-2)|df|^2)) \\
& +4(n-1) \left( k_1(\Delta_g k_2+(n-2)\langle df,dk_2\rangle + k_2(\Delta_g k_1+(n-2)\langle df,dk_1\rangle \right) \\
& -2(n-1)(n-2)\langle dk_1, dk_2 \rangle
\end{split}
\]
We then get
\begin{equation}
    \label{equation:inter}
\begin{split}
\|\partial^2_f\Phi(f,h)\cdot(k_1,k_2)\|_{C^{m,\alpha}} & \leq C(\|e^{-2f}\|_{C^{m,\alpha}}\|R(f,k_1,k_2)\|_{C^0} + \|e^{-2f}\|_{C^{0}}\|R(f,k_1,k_2)\|_{C^{m,\alpha}}) \\
& \leq C((1+\|f\|_{C^{m,\alpha}})\|R(f,k_1,k_2)\|_{C^0} + \|R(f,k_1,k_2)\|_{C^{m,\alpha}})
\end{split}
\end{equation}
where $C := C(m,\alpha,\|f\|_{C^0})$ and the first inequality follows from the bound on products of Hölder functions \cite[Chapter 2, Proposition A.2.1.1]{Alinhac-Gerard-07}, the second from the interpolation argument of \cite[Example 2.2.4]{Hamilton-82} to bound tamely exponential terms. If we make the assumption that $\|f\|_{C^2} < 1$ is uniformly bounded, then $C$ is now independent of $f$. Applying once again \cite[Chapter 2, Proposition A.2.1.1]{Alinhac-Gerard-07}, and using $\|f\|_{C^2} < 1$, we easily get that \eqref{equation:inter} is bounded by:
\[
\begin{split}
\|\partial^2_f\Phi(f,h)\cdot(k_1,k_2)\|_{C^{m,\alpha}} & \leq C(\|k_1\|_{C^2}\|k_2\|_{C^2}(1+\|f\|_{C^{m+2,\alpha}}) \\
&+ \|k_1\|_{C^2}\|k_2\|_{C^{m+2,\alpha}} + \|k_2\|_{C^2}\|k_1\|_{C^{m+2,\alpha}}),
\end{split}
\]
which is precisely the content of \cite[Condition (H1), p145]{Alinhac-Gerard-07}. Moreover, \cite[Condition (H2), p145]{Alinhac-Gerard-07} is also satisfied, that is the linear map $k \mapsto \partial_f\Phi(f,h)k$ computed in \eqref{equation:dphi} is also invertible and the inverse satisfies tame estimates, since it is an elliptic principally selfadjoint differential operator of order $2$ obtained by perturbation of $L$ (which is assumed to be invertible). Hence, the Nash-Moser theorem applies and provides the expected regularity statement.
\end{proof}

We now prove that the \textbf{(LPSC)} condition is a generic one:

\begin{proposition}\label{proposition: perturb to get LPSC}
Let $(\Sigma,g)$ be a smooth compact Riemannian manifold with boundary. Then for all $m \geq 0$, $\eps > 0$, there exists $f \in C^\infty_{\mathrm{comp}}(\Sigma^\circ)$ such that $\|f\|_{C^m} < \eps$ and $e^{2f}g$ is \emph{\textbf{(LPSC)}} on $\Sigma$.
\end{proposition}

For the sake of simplicity, we now use the notation $L_f$ in place of $L_{e^{2f}g}$ as all the perturbations will be conformal.

\begin{proof}
The claim can be reformulated as follows: for all $m \geq 0$,$\delta > 0$, there exists $f \in C^\infty_{\mathrm{comp}}(\Sigma^\circ)$ such that $\|f\|_{C^m} < \delta$ and the Dirichlet Laplacian $L_f := (n-1)\Delta_{e^{2f}g} + s_{e^{2f}g}$ has trivial kernel. Using the formulas for conformal change of $\Delta_g$ and $s_g$, we get
\[
L_f = e^{-2f}L_0 + e^{-2f}(n-1)((n-2)g(df, d\cdot) - 2\Delta_g f - (n-2)|df|^2).
\]
Differentiating with respect to $f$ at $f=0$ in the direction of $k \in C^\infty_{\mathrm{comp}}(\Sigma^\circ)$, we get
\begin{equation}
\label{equation:ldot}
\dot{L}_0 = -2 k L_0 + (n-1)(n-2)g(dk, d\cdot) - 2(n-1)\Delta_g k.
\end{equation}
(Here and below, we use the dot notation for $\partial_f \cdot|_{f=0}$.) Let $\gamma$ be a small contour in $\C$ around $0$ such that $0$ is the only eigenvalue of $L_0$ inside $\gamma$ (and avoiding all other eigenvalues of $L_0$). As $L_f$ depends continuously on $f$ in the $C^{2,\alpha}$-topology ($0 < \alpha < 1$), for $f$ small enough in $C^{2,\alpha}(\Sigma)$, $\gamma$ will also avoid all eigenvalues of $L_f$. Let $\lambda_f$ be the sum of the eigenvalues of $L_f$ inside the contour $\gamma$, that is
\begin{equation}
    \label{equation:integral}
\lambda_f := \Tr(L_f\Pi_f), \qquad \Pi_f := \dfrac{1}{2i\pi} \oint_\gamma (z-L_f)^{-1} \dd z.
\end{equation}
Here, the operator $\Pi_f$ is the spectral projector onto the eigenspaces associated to the eigenvalues inside $\gamma$.

Since $\Pi_f$ depends continuously on $f$, its rank is constant (for small $f$) and equal to the sum of the dimension of the eigenspaces of $L_f$ associated to eigenvalues inside $\gamma$. Our aim is to produce a small perturbation $f \neq 0$ such that $\lambda_f \neq 0$. This will imply that the $0$-eigenvalue of $L_0$ was ``split'' into at least two distinct eigenvalues, one of them being non-zero; in other words, $\dim \ker L_f < \dim \ker L_0$. Repeating this process a finite number of times, we will be able to produce a perturbation $f$ such that $\ker L_f = \{0\}$.

Let $(u_i)_{i=1}^N$ be an orthonormal basis of $\ker L_0$ in $L^2(\Sigma)$. Since $L_0$ is real (that is, it maps real-valued functions to real-valued functions), we can further choose the $u_i$'s to be real-valued functions. The following formula holds:

\begin{claim}
Define $\theta := \sum_{i=1}^N u_i^2$. Then one has :
\begin{equation}
    \label{equation:lambdadot}
\dot{\lambda}_0=- (n-1)(\tfrac{n}{2}+1) \int_{\Sigma} \Delta_g k(x) \cdot\theta(x) ~\dd\vol_g(x).
\end{equation}
\end{claim}

\begin{proof}
Using \eqref{equation:integral} and differentiating under the integral sign, one gets $\dot{\lambda}_0 = \Tr(\Pi_0 \dot{L}_0 \Pi_0)$. Hence, using \eqref{equation:ldot}, we have:
\[
\begin{split}
\dot{\lambda}_0 & = \sum_{i=1}^N \langle \dot{L}_0 u_i, u_i \rangle_{L^2} = (n-1) \sum_{i=1}^N (n-2) \langle dk \cdot du_i, u_i \rangle_{L^2} - 2 \langle u_i \Delta k, u_i \rangle_{L^2} \\
& = (n-1) \sum_{i=1}^N \tfrac{n-2}{2} \langle dk, d(u_i^2) \rangle_{L^2} - 2 \langle \Delta k, u_i^2 \rangle_{L^2} \\
&= -(n-1) (\tfrac{n}{2}+1) \int_{\Sigma} \Delta_g k(x) \cdot \theta(x) ~\dd\vol_g(x),
\end{split}
\]
where we used in the second equality that $L_0u_i = 0$, in the third that the $u_i$'s are real-valued, in the fourth that $k$ vanishes to order $1$ on the boundary $\partial \Sigma$ (integration by parts).
\end{proof}

Now, we claim that there exists $k \in C^\infty_{\mathrm{comp}}(\Sigma^\circ)$ such that $\dot{\lambda}_0 \neq 0$. Indeed, if not, then we obtain after integration by parts in \eqref{equation:lambdadot} that
\[
\int_{\Sigma} k(x) \cdot\Delta_g\theta(x) ~\dd\vol_g(x) = 0, \qquad \forall k \in C^\infty_{\mathrm{comp}}(\Sigma^\circ).
\]
This implies that $\Delta_g\theta = 0$ on $\Sigma$ but $\theta = 0$ on $\partial \Sigma$ (as the $u_i$'s are eigenfunctions of the Dirichlet Laplacian $L_0$), so this would force $\theta \equiv 0$, which is a contradiction. Hence, we can find some $k_1 \in C^\infty_{\mathrm{comp}}(\Sigma^\circ)$ such that $\dot{\lambda}_0 \neq 0$. 

As a consequence, for $s \in (-1,1)$ small enough, $s \mapsto \lambda_{sk_1} \neq 0$. We then fix $s_1 >0$ small such that $f_1 := s_1k_1$ satisfies $\|f_1\|_{C^m} < \eps/K$ where $K := \dim \ker L_0$; the metric $e^{2f_1}g$ satisfies by construction that $\dim \ker L_{f_1} < \dim \ker L_0$. Iterating the same perturbation process (at most) $K$ times, we thus obtain a perturbation $e^{2f}g = e^{2(f_1+...+f_K)}g$ such that $\ker L_{f} = 0$ and $\|f\|_{C^m} < \eps$. This concludes the proof.
\end{proof}





\subsection{Generic metrics on manifolds with boundary}

If $g$ is a metric, we denote by $\mathrm{Isom}(g)$ its isometry group. Following Ebin's work (see \cite[Theorem 147]{Ebin-68}), it is well-known that generic metrics on closed manifold have trivial isometry group (i.e. the set of metrics whose isometry group is trivial is open and dense). We shall need a similar (though slightly weaker) result for manifolds with boundary:

\begin{lemma}
\label{lemma:kill-isometries}
Let $(\Sigma,g)$ be a smooth compact manifold with boundary. Then for all $\alpha \in (0,1)$, for all $\eps >0$ small enough, there exists $\delta > 0$ and $f \in C^\infty_{\mathrm{comp}}(\Sigma^\circ)$, such that $\|f\|_{C^{3,\alpha}} < \eps$, the group $\mathrm{Isom}(e^{2f}g)$ is trivial, and for all $\varphi \in C^4(\Sigma)$ with $\|\varphi\|_{C^4} < \delta$, $\mathrm{Isom}(e^{2(f+\varphi)}g)$ is also trivial.
\end{lemma}

The previous lemma shows that the set $\mc{G}$ of metrics without isometries contains a dense open subset. Following Ebin's proof of \cite[Theorem 147]{Ebin-68}, one could actually show that $\mc{G}$ itself is open and dense but this will not be needed in what follows. Lemma \ref{lemma:kill-isometries} is probably classical but since we could not find a proof in the literature, we provide a quick one.


\begin{proof}[Proof of Lemma \ref{lemma:kill-isometries}]
First, by Proposition \ref{proposition: perturb to get LPSC}, there exists $f \in C^\infty_{\mathrm{comp}}(\Sigma^\circ)$ with $f$ arbitrarily small in the $C^{3,\alpha}$-topology such that $e^{2f}g$ is \textbf{(LPSC)} on $\Sigma$. Hence, up to renaming $g$ by $e^{2f}g$, we can already assume that $g$ is \textbf{(LPSC)}.

For $\delta > 0$, set $\Sigma_\delta := \left\{x \in \Sigma ~|~ d_g(x,\partial \Sigma) > \delta\right\}$. In the following, we fix $\delta_0 := \tfrac{1}{1789}\mathrm{diam}_g(M)$ sufficiently small. 
We apply Lemma \ref{lemma: LPSC property} to $(\Sigma,g)$ with $m=1$ and $0 < \alpha < 1$ arbitrary; let $C,\eps_0$ be as in Lemma~\ref{lemma: LPSC property}. Now, take $h \in C^\infty_{\mathrm{comp}}(\Sigma^\circ)$ such that $\|h\|_{C^{1,\alpha}}<\eps/C$ and $s_{g}+h$ is a Morse function in $\Sigma_{\delta_0}$, that is $s_{g}+h$ has isolated critical points (if any) in $\Sigma_{\delta_0}$ with non-degenerate Hessian (this is always possible by density of Morse functions in the smooth topology). We further require $s_{g}+h$ to have at least $n+1$ critical points $\{x_i\}_{i=0,\ldots, n}$ so that they are at distance $> \tfrac{1}{496}\mathrm{diam}_{g}(\Sigma)>\delta_0$ from $\partial \Sigma$ (with respect to the metric $g$), and $\{\dot{\gamma}^g_{x_0x_i}(0)\}_{i=1,\ldots, n}$ span $T_{x_0}\Sigma$ where $\gamma^g_{x_0x_i}$ is a unique geodesic with respect to $g$ connecting points $x_0$ and $x_i$ and $\gamma^g_{x_0x_i}(0)=x_0$.


Since the set of critical points is finite, we can further guarantee that the perturbation satisfies that all the eigenvalues at all critical points are different, that is: for all $x \in \mathrm{Crit}(s_{g}+h|_{\Sigma_{\delta_0}})$, the eigenvalues $\lambda_1(x) < ... < \lambda_n(x)$ of $\mathrm{Hess}_x(s_{g}+h)$ are different, and for all $x, y \in \mathrm{Crit}(s_{g}+h|_{\Sigma_{\delta_0}})$, there is no $1 \leq i,j \leq n$ such that $\lambda_i(x)=\lambda_j(y)$ (note that this can be achieved by keeping the $C^2$-norm of $h$ small, hence its $C^{1,\alpha}$ small as well).

By Lemma \ref{lemma: LPSC property}, there exists $f \in C^\infty_0(\Sigma)$ (vanishing to order $0$ on $\partial \Sigma$) such that $s_{e^{2f}g} = s_{g} + h$, and $\|f\|_{C^{3,\alpha}} \leq C\|h\|_{C^{1,\alpha}} < \eps$. Let $\psi \in C^\infty_{\mathrm{comp}}(\Sigma^\circ)$ such that $\psi \equiv 1$ in a neighborhood of $\Sigma_{\delta_0}$. Note that, up to changing $\eps$ by $C_0\eps$ for some $C_0$ depending only on $\delta_0$, we can still guarantee $\|f \psi\|_{C^{3,\alpha}} < \eps$. Moreover, the distance of the points $x_i$'s from $\partial \Sigma$ with respect to the new metric $e^{2f\psi}g$ is $> \delta_0$ and, up to taking $\eps > 0$ small enough, $\{\dot{\gamma}^{e^{2f\psi}g}_{x_0x_i}(0)\}_{i=1,\ldots, n}$ span $T_{x_0}\Sigma$.

We now claim that $e^{2f\psi}g$ has no isometry. First, observe that on $\Sigma_{\delta_0}$, $s_{e^{2f\psi}g} = s_{e^{2f}g} = s_g + h$. Let $x$ be a critical point of $s_{e^{2f \psi}g}$ at distance $> \delta_0$ from $\partial \Sigma$ with respect to $e^{2f\psi}g$. If $\phi$ is an isometry of $e^{2f \psi}g$ on $\Sigma$, then it maps critical points of $s_{e^{2f \psi}g}$ to critical points of $s_{e^{2f \psi}g}$ so $x$ is mapped to another critical point $x'$; but $d_{e^{2f\psi}g}(x,\partial \Sigma) = d_{e^{2f\psi}g}(x',\partial \Sigma) > \delta_0$, so we know that $x'$ is not contained in $\Sigma \setminus \Sigma_{\delta_0}$. Moreover, $\phi$ maps isometrically $\mathrm{Hess}_{x}$ to $\mathrm{Hess}_{x'}$. However, if $x \neq x'$, we would get a contradiction from the fact that the Hessians of $s_g + h$ have different eigenvalues at critical points in $\Sigma_{\delta_0}$ by construction, so $x = x'$. Applying this argument with $x=x_i$, we get that $x_i$ are fixed points of $\phi$ so the geodesics $\gamma^{e^{2f\psi}g}_{x_0x_i}$ are also fixed by $\phi$. Since $\{\dot{\gamma}^{e^{2f\psi}g}_{x_0x_i}(0)\}_{i=1, \ldots,n}$ span $T_{x_0}\Sigma$, we obtain that $d\phi_{x_0}=\mathrm{id}$ so $\phi=\mathrm{id}$ on $\Sigma$. This proves the first part of the Lemma.

To see that $e^{2(\varphi + f)}g$ has no isometry if $\|\varphi\|_{C^4}$ is small enough, it suffices to observe that $s_{e^{2\varphi + f}g} - s_{e^{2f}g}$ is then small in $C^2$, and thus the arguments of the previous paragraph involving the Hessian still carry on.


\end{proof}


\subsection{Perturbation lemma}

The proof of Theorem \ref{theorem:transfer} will also be based on the following lemma:

\begin{proposition}
\label{proposition:technical}
Let $(M,g)$ be a smooth closed manifold, and let $\Sigma \subset M$ be compact codimension $0$ submanifold with boundary. For all $m \geq 2$, $\eps > 0$, there exists $f \in C^m_{0;m}(M \setminus \Sigma)$ with support in $M \setminus \Sigma$ and vanishing to order $m$ at $\partial \Sigma$ such that $\|f\|_{C^m} < \eps$ and the following property holds: for all $x \in M \setminus \Sigma^\circ, y \in \Sigma$, there exists no open neighborhood $U_x$ of $x$, and $U_y$ of $y$, and no diffeomorphism $\phi \colon U_x \cap (M \setminus \Sigma^\circ) \to U_y \cap \Sigma$ such that $\phi^*(g|_{U_y \cap \Sigma}) = e^{2f}g|_{U_x \cap (M \setminus \Sigma^\circ)}$.
\end{proposition}

In other words, we can perturb the metric $g$ so that there exists no local isometries between $\Sigma$ and its complement $M \setminus \Sigma^\circ$. The notation $C^m_{0;m}(M \setminus \Sigma)$ is used to denote functions supported in $M \setminus \Sigma$ and vanishing to order $m$ at $\partial \Sigma$. As we shall see from the proof, the same statement holds with $g$ that is $C^{m+1}$-regular instead of smooth. The space $C^m_{0;m}(M \setminus \Sigma)$ endowed with the norm $\|\cdot\|_{C^m}$ is a Banach space; the proof is based on a Baire category type of argument in $C^m_{0;m}(M \setminus \Sigma)$. In the proof, the $C^m$-norms are computed with respect to the metric $g$ for the sake of simplicity, but the statement holds regardless of the metric considered to compute the norm. 

\begin{proof}
It suffices to prove the statement for $m \geq 2$ odd, which we will assume throughout the proof. Define for $k \geq 1$
\begin{equation}
\label{equation:omega-k}
\begin{split}
\Omega_k := & \left\{f \in C^m_{0;m}(M \setminus \Sigma) ~|~ \forall x \in M \setminus \Sigma^\circ, y \in \Sigma, \right. \\
& \qquad \left. \forall \phi \in  \mathrm{Diff}(B_{e^{2f}g}(x,1/k) \cap (M \setminus \Sigma^\circ),B_{g}(y,1/k) \cap \Sigma), \right. \\
& \qquad \left.  \phi^*(g|_{B_{g}(y,1/k) \cap \Sigma}) - e^{2f}g \neq 0, \text{on } B_{e^{2f}g}(x,1/k) \cap (M \setminus \Sigma^\circ)\right\},
\end{split}
\end{equation}
where $B_g(x,\eps)$ is the closed geodesic ball (with respect to the metric $g$), centered at $x$ and with radius $\eps > 0$. In other words, $\Omega_k$ is the set of functions supported in $M \setminus \Sigma$ such that the metric $e^{2f}g$ has no isometries between any geodesic ball of radius $1/k$ centered in $\Sigma$, and any other geodesic ball of radius $1/k$ centered in $M \setminus \Sigma^\circ$. Note that the regularity of the diffeomorphism $\phi$ in the definition \eqref{equation:omega-k} is not specified (it is at least $C^1$) but since the metrics are $C^m$, any isometry is automatically $C^{m+1}$ by \cite[Theorem 2.1]{Taylor-06}.

\begin{claim}
\label{claim:1}
    For all $k \geq 1$, $\Omega_k$ is open in the $C^m$-topology.
\end{claim}

\begin{proof}Assume for a contradiction that $f \in \Omega_k$ and $f_n \in C^m_{0;m}(M\setminus\Sigma)$ is a sequence such that $f_n \to f$ in the $C^m$-topology but $f_n \notin \Omega_k$. Then, for every $n \geq 0$, there exists $x_n \in M \setminus \Sigma^\circ$, $y_n \in \Sigma$ and an isometry $\phi_n \colon B_{e^{2f_n}g}(x_n,1/k) \cap (M \setminus \Sigma^\circ) \to B_{g}(y_n,1/k) \cap \Sigma$ such that $\phi^*(g|_{B_{g}(y_n,1/k) \cap \Sigma}) = e^{2f_n}g$ on $B_{e^{2f_n}g}(x_n,1/k) \cap (M \setminus \Sigma^\circ)$. By compactness, up to considering subsequences, we have that $x_n \to x$ in $M \setminus \Sigma^\circ$ and $y_n \to y$ in $\Sigma$. The diffeomorphisms $\phi_n$ have uniformly (in $n \geq 0$) bounded $C^1$-norm when computed with respect to the fixed smooth metric $g$; moreover, they are isometries with respect to the metric $e^{2f_n}g$. As a consequence, by the Arzelà-Ascoli compactness theorem, we obtain that, up to taking another subsequence, $\phi_n \to \phi$ and $\phi_n^{-1} \to \phi^{-1}$ as $n \to \infty$ in the $C^0$-topology, where $\phi \colon B_{e^{2f}g}(x,1/k) \cap (M \setminus \Sigma^\circ) \to B_{e^{2f}g}(y,1/k) \cap \Sigma$ is a homeomorphism. Moreover, as $\phi_n$ are distance-preserving homeomorphisms, we deduce that $\phi$ is also a distance-preserving homeomorphism. Since $e^{2f}g$ is $C^m$, we deduce by \cite[Theorem 2.1, items 1 and 4]{Taylor-06} that $\phi$ is $C^{m+1}$, which then implies that $f \notin \Omega_k$. This is a contradiction so $\Omega_k$ is open. \end{proof}

We now prove that $\Omega_k \cap C^\infty(M \setminus \Sigma)$ is dense in smooth (compactly supported) functions for the $C^m$-topology in the following sense: 

\begin{claim}
\label{claim:2}
    For all $k \geq 1, \eps > 0$ and $f_0 \in C^\infty_{\mathrm{comp}}(M \setminus \Sigma)$, there exists $f \in C^\infty_{\mathrm{comp}}(M \setminus \Sigma)$ such that $f_0+f \in \Omega_k$ and $\|f\|_{C^{m}} < \eps$.
\end{claim}

It will be important in the proof that $g$ and $f_0$ are smooth (or at least $C^{m+1}$).

\begin{proof}First, up to changing $g$ by $e^{2f_0}g$, it suffices to prove this claim for $f_0=0$. Second, using Proposition \ref{proposition: perturb to get LPSC}, we can perform a perturbation $e^{2f}g$ of $g$ so that $e^{2f}g$ is \textbf{(LPSC)} and $f \in C^\infty_{\mathrm{comp}}(M \setminus \Sigma)$ has arbitrarily small $C^{m}$-norm. For the sake of simplicity, we still denote by $g$ the new perturbed metric which is \textbf{(LPSC)}.

Consider a set of distinct points $(x_i)_{i=1}^N$ in $M$ with the following properties:
\begin{enumerate}[label=(\roman*)]
    \item It is $\frac{1}{10k}$-dense with respect to the metric $g$, that is for every $x \in M$, there exists $1 \leq i \leq N$ such that $d_{g}(x,x_i) < \frac{1}{10k}$;
    \item For every $1 \leq i \leq N$, $d_{g}(x_i, \partial \Sigma) > \frac{1}{100k}$.
\end{enumerate}
Let $\eta := \min_{i \neq j} d_g(x_i, x_j)$. Note that the existence of such points $(x_i)_{i=1}^N$ is robust with respect to small $C^m$-perturbations of the metric $g$. Up to rearranging the $x_i$'s, we can further assume that $x_1, ..., x_{N_0} \in M \setminus \Sigma$ and $x_{N_0+1},...,x_N \in \Sigma$. 

Let $\chi \in C^\infty_{\mathrm{comp}}(\R)$ be a smooth even nonnegative function with support in $(-1,1)$. For $\delta > 0$ small enough (to be determined later), we set:
\begin{equation}
    \label{equation:def}
\chi_i(x) := \delta^{m-2+\frac{1}{m}} \chi\left(\dfrac{d_g(x,x_i)}{\delta}\right), \qquad h_{\delta} := \sum_{i=1}^{N_0} \chi_i.
\end{equation}
The function $h_\delta$ is a bump function with bumps localized at the points $x_i \in M \setminus \Sigma^\circ$ and supported in a $g$-geodesic neighborhood of size $\delta$ of these points. Note that $h_\delta$ is smooth as $\chi$ is even. Taking $\delta$ small enough (depending on $k$ and $\eta$), we get from $\eta$-separation of the $x_i$'s that $h_\delta(x) = \chi_i(x)$ for all $x \in B_g(x_i,\delta)$. In particular, using \eqref{equation:formula}, this yields that for all $X \in T_{x_i}M$ such that $|X|_g = 1$, $1 \leq i \leq N_0$ and $j \geq 0$:
\begin{equation}
    \label{equation:nablaj}
    \nabla^j_g h_\delta(x_i)(X,...,X) = \partial_t^j h_\delta(\gamma(t))|_{t = 0} = \delta^{m-2-j+\frac{1}{m}} \chi^{(j)}(0),
\end{equation}
where $\gamma(t) := \exp_{x_i}(tX)$.

We now make use of the assumption that $m$ is odd, and impose the condition $\chi^{(m-1)}(0) = -1, \chi^{(j)}(0) = 0$ for $j \leq m-2$. We claim that there exists a constant $C := C(g,k)>0$ such that:
\begin{equation}
    \label{equation:bounds}
    \|h_{\delta}\|_{C^{m-2}} \leq C \delta^{\frac{1}{m}}, \quad \|h_{\delta}\|_{C^{m-1}} \leq C \delta^{\frac{1}{m}-1}, \quad |\nabla^{m-1}_{g} h_{\delta}(x_i)|_g \geq C \delta^{\frac{1}{m}-1}~~ (1 \leq i \leq N_0),
\end{equation}
where we recall that all $C^k$-norms are computed with respect to $g$ (as introduced in \S\ref{ssection:functional-spaces}). Indeed, the first two inequalities in \eqref{equation:bounds} follow from the definition \eqref{equation:def} while the third follows from \eqref{equation:nablaj}.



By interpolation, we therefore obtain $\|h_{\delta}\|_{C^{m-2,\alpha}} \leq C \delta^{\frac{1}{m}-\alpha}$. We take $\alpha := \frac{1}{2m}$ and apply Lemma \ref{lemma: LPSC property}. This provides $f_{\delta} \in C^\infty_{0}(M\setminus \Sigma)$ (vanishing to order $0$ at $\partial \Sigma$) such that
\begin{equation}
    \label{equation:prescription}
s_{e^{2f}g} = s_g + h_\delta, \qquad \|f_{\delta}\|_{C^{m,\alpha}} \leq C \|h_{\delta}\|_{C^{m-2,\alpha}} \leq C \delta^{\frac{1}{m}-\alpha} = C \delta^{\frac{1}{2m}},
\end{equation}
where $C:=C(g,k) > 0$ is independent of $\delta$.

We then multiply by a cutoff function in order to obtain functions with compact support in $M \setminus \Sigma$. By assumption, $d(x_i,\partial \Sigma) > \frac{1}{100k}$ for all $1 \leq i \leq N$ and $\chi_i$ is supported in a ball of radius $\delta$ around $x_i$; hence, taking $\delta < \frac{1}{1000k}$, we can guarantee that $h_\delta$ is supported in
\[
V := \left\{x \in M\setminus \Sigma^{\circ} ~|~ d(x,\partial \Sigma) > \tfrac{9}{10} \cdot \tfrac{1}{100k} \right\}.
\]
Now, letting $\psi \in C^\infty_{\mathrm{comp}}(M\setminus \Sigma^\circ)$ be a function such that $\psi \equiv 1$ on $V$ and $\psi$ is supported in $V' := \left\{x \in M\setminus \Sigma^{\circ} ~|~ d(x,\partial \Sigma) > \tfrac{4}{5} \cdot \frac{1}{100k} \right\}$, we see that $f_{\delta} \psi$ satisfies a similar estimate to \eqref{equation:prescription} (with a possibly different constant $C := C(g,k) >0$), that is
\begin{equation}
\label{equation:control}
\|f_{\delta} \psi\|_{C^{m,\alpha}} \leq  C \delta^{\frac{1}{2m}}.
\end{equation}
Moreover, $s_{e^{2 f_{\delta} \psi} g} = s_{e^{2f_{\delta}}} = s_g + h_\delta$ on $V$. We then take $\delta > 0$ small enough such that:
\begin{enumerate}[label=(\roman*)]
    \item $C \delta^{\frac{1}{2m}} < \eps$, which guarantees by \eqref{equation:control} that $\|f_\delta \psi\|_{C^{m,\alpha}} < \eps$;
    \item $\tfrac{1}{2} C \delta^{\frac{1}{m}-1} > \sup_{x \in \Sigma} |\nabla^{m-1}_g s_g(x)| + 1/2 \cdot \sup_{x_i \in M \setminus \Sigma^{\circ}} |\nabla^{m-1}_g s_g(x_i)|$.
\end{enumerate}
We claim that this choice ensures $f_\delta \psi \in \Omega_k$. Indeed, assume for a contradiction that there exists an isometry
\[
\phi\colon (B_{e^{2f_\delta \psi}g}(x,1/k) \cap (M \setminus \Sigma^\circ), e^{2f_\delta \psi}g) \to (B_{g}(y,1/k) \cap \Sigma,g).
\]
For $\delta$ small enough, the points $(x_i)_{1 \leq i \leq N}$ still satisfy (i-ii) above for the metric $e^{2f_\delta \psi}g$; hence, there exists a point $x_i \in B_{e^{2f_\delta \psi}g}(x,1/k) \cap (M \setminus \Sigma^\circ)$. Let $y_i := \phi(x_i)$. Then $|\nabla^{m-1}_{e^{2f_\delta \psi}g} s_{e^{2f_\delta \psi}g}(x_i)|_{e^{2f_\delta \psi}g} = |\nabla^{m-1}_g s_g (y_i)|_g$. But the following holds:
\[
\begin{split}
|\nabla^{m-1}_g s_g (y_i)|_g & = |\nabla^{m-1}_{e^{2f_\delta \psi}g} s_{e^{2f_\delta \psi}g}(x_i)|_{e^{2f_\delta \psi}g} \\
& = |\nabla^{m-1}_{e^{2f_\delta}g} s_{e^{2f_\delta}g}(x_i)|_{e^{2f_\delta}g}  \\
& = |\nabla^{m-1}_{e^{2f_\delta}g} (s_g + h_\delta)(x_i)|_{e^{2f_\delta}g} \\
&\geq 1/2 \cdot |\nabla^{m-1}_{g} (s_g + h_\delta)(x_i)|_{g} \\
& \geq 1/2 \cdot C \delta^{\frac{1}{m}-1} - 1/2 \sup_{x_i \in M \setminus \Sigma^\circ} |\nabla^{m-1}_{g} s_g(x_i)|_g  > \sup_{x \in \Sigma} |\nabla^{m-1}_{g} s_g(x)|_g.
\end{split}
\]
where the first inequality follows from Lemma \ref{lemma:comparison}, the second from \eqref{equation:bounds}, and the third from item (ii) above in the choice of $\delta$. This is a contradiction so $f_\delta \psi \in \Omega_k$. \end{proof}

We can then complete the proof of Proposition \ref{proposition:technical} by applying a Baire category type of argument to the intersection $\cap_{k \geq 1} \Omega_k$. 

\end{proof}

Note that, since the construction in Claim \ref{claim:2} consists in making the $C^{m-1}$-norm of $s_{e^{2f}g}$ blow-up, the $C^{m+1}$ norm of $e^{2f}g$ blows-up, and thus the function $f \in C^m_{0;m}(M\setminus \Sigma)$ provided by Proposition \ref{proposition:technical} will clearly not be $C^{m+1}$-regular.

\section{Proof of main results}

\subsection{Proof of the transfer principle}

We now prove Theorem \ref{theorem:transfer}. We start with the following preliminary lemma:

\begin{lemma}
\label{lemma:same-mls}
Let $\Sigma$ be a smooth compact connected oriented manifold with boundary. Let $g_1$ and $g_2$ be two metrics of Anosov type on $\Sigma$ such that $d_{g_1} = d_{g_2}$, and further assume that $g_1$ and $g_2$ are consistently extendable to Anosov metrics $g_1'$ and $g_2'$ on a closed manifold $M$. Then the two extended metrics have same marked length spectrum on $M$, that is $L_{g'_1} = L_{g'_2}$.
\end{lemma}

\begin{proof}
Denote by $\mc{C}$ the set of free homotopy classes on $M$, let $c \in \mc{C}$ and denote by $\gamma_c^1$ the $g'_1$-geodesic representative of $c$. If $\gamma_c^1$ never intersects $\partial \Sigma$, then it is either contained in $\Sigma^\circ$ or in $\left(M\setminus \Sigma\right)^\circ$. If $\gamma_c^1 \cap \partial \Sigma \neq \emptyset$, then there are two cases: either there exists a transverse intersection point, in which case $\gamma_c^1 \cap \Sigma^\circ \neq \emptyset$ and $\gamma_c^1 \cap \left(M \setminus \Sigma\right)^\circ \neq \emptyset$, or all intersection points of $\gamma_c^1 \cap \partial \Sigma$ are tangential intersections, in which case $\gamma_c^1 \subset M \setminus \Sigma^\circ$ by strict convexity of the boundary $\partial \Sigma$ with respect to $g_1$. Hence, there are three cases to consider: (i) if $\gamma_c^1$, the $g'_1$-geodesic representative of $c$, is contained in $M\setminus
\Sigma^\circ$, (ii) if it intersects the boundary of $\Sigma$ and at least one of the intersection points is transverse, (iii) if it is contained in $\Sigma$. \\

\textbf{Case (i).} Since $g_1$ and $g_2$ are consistently extendable and there is a unique geodesic representative for each $c\in\mc{C}$, the result is immediate as the geodesic representatives of $c$ for $g'_1$ and $g'_2$ coincide on $M \setminus \Sigma^\circ$. \\

\textbf{Case (ii).} We now assume that $\gamma^1_c$ intersects non-trivially both $\Sigma^\circ$ and $\left(M \setminus \Sigma\right)^\circ$. Let $\ell$ be the $g_1'$-length of $\gamma^1_c$ and choose a parametrization $\gamma^1 : [0,\ell] \to M$ such that $x := \gamma^1_c(0) \in M \setminus \Sigma$ (we drop the subscript $c$ for simplicity). Let $v := \dot{\gamma}^1(0) \in SM_1$. By our assumption and the strict convexity of $\partial \Sigma$ with respect to $g_1$, there are times $0 < t_1^- < t_1^+ < t_2^- < t_2^+ < ... < t_N^- < t_N^+ < \ell$ such that for all $1 \leq i \leq N$, $\gamma^1(t_i^-), \gamma^1(t_i^+) \in \partial \Sigma$ and $\gamma^1((t_i^-,t_i^+)) \subset \Sigma^\circ$. The point $(x,v)$ is also in $SM_2$, the unit tangent bundle of $g_2'$, since $g_1'$ and $g_2'$ are equal on $M \setminus \Sigma^\circ$ (see Figure \ref{figure: case ii}).

We claim that the $g_2'$-geodesic generated by $(x,v) \in S\Sigma_2$ is closed, of length $\ell$ and in the homotopy class $c$. For the sake of clarity, let us run the argument with $N=1$, the generalization to $N \geq 2$ being immediate. Denote by $\gamma^2 : [0,\ell] \to M$ the $g_2'$-geodesic segment (of length $\ell$) generated by $(x,v) \in SM_2$. Since $g_1'=g_2'$ on $M\setminus \Sigma^\circ$, it is clear that $\gamma^1(t)=\gamma^2(t)$ for all $0 \leq t \leq t_1^-$. In particular, $\gamma^1(t_1^-)=\gamma^2(t_1^-)=:p \in \partial \Sigma$ and $\dot{\gamma}^1(t_1^-)=\dot{\gamma}^2(t_1^-)=:w \in T_p\Sigma$. Since $g_1$ and $g_2$ are consistently extendable, their $C^\infty$-jets agree on $\partial \Sigma$; in particular, we can apply Lemma \ref{lemma:lens} as $g_1=g_2$ on $T_{\partial \Sigma} \Sigma \times T_{\partial \Sigma} \Sigma$. We thus obtain that the $g_1'$-geodesic and $g_2'$-geodesic generated by $(p,w)$ in $\Sigma$ have same length $\ell_1 := t_1^+-t_1^-$, exit $\Sigma$ at the same point $q \in \partial \Sigma$ with same unit speed vector $z \in T_q\Sigma$, and belong to the same homotopy class of curves joining $p$ to $q$. In other words, $\gamma^1(t_1^+) = \gamma^2(t_1^+) = q$, $\dot{\gamma}^1(t_1^+) = \dot{\gamma}^2(t_1^+) = z$, and the geodesic segments $\gamma^1([t_1^-,t_1^+])$ and $\gamma^2([t_1^-,t_1^+])$ are homotopic (via a homotopy fixing their endpoints). Using once again that $g_1'=g_2'$ on $M \setminus \Sigma^\circ$, we deduce that $\gamma^1(t) = \gamma^2(t)$ for all $t_1^+ \leq t \leq \ell$. Moreover, since $\gamma^1$ closes up, that is $\dot{\gamma}^1(0) = \dot{\gamma}^1(\ell)$, we get that $\gamma^2$ closes up too, so $\gamma^2$ is a closed $g_2'$-geodesic of length $\ell$ (and not shorter). Finally, $\gamma^2$ is in the same free homotopy class as $\gamma^1$ (since it coincides with $\gamma^1$ on $M \setminus \Sigma$ and is in the same homotopy class on $\Sigma$) and this shows that $\gamma^2$ is the unique closed $g_2'$-geodesic in $c$. The generalization of this argument to $N \geq 2$ is immediate. This proves the claim. 

\begin{figure}[H]
      \centering
      \begin{subfigure}{.5\textwidth}
            \centering
			\begin{tikzpicture}[x={(1cm,0cm)},y={(-.2cm,-.2cm)},z={(0cm,1cm)},scale=0.9]

                \node[] at (-2,0,1.8) {\large $\Sigma$};
                \node[] at (5.3,1,2.3) {\large $M\setminus\Sigma$};
    			\draw[thick] (0,0,-0.5) to[bend left] (0,0,0.5) node at (0.2,0,0.8) {\large$\partial\Sigma$};
				\draw[dashed,thick] (0,0,-0.5) to[bend right] (0,0,0.5);
				\draw[rotate=270] (-.5,0,0) to[out=-90,in=90] (-1.5,0,-2);
				\draw[rotate=270] (-1.5,0,-2) to[out=-90, in=-90] (1.5,0,-2);
				\draw[rotate=270] (.5,0,0) to[out=-90,in=90] (1.5,0,-2);

                \draw [line width=0.3mm, red] plot [smooth, tension=.7] coordinates {(-.45,-1,-.05)(-2,-1,1)(-2.8,-1,0)(-2,-1,-1)(-0.05,1,0)};
                \draw [line width=0.3mm, red] (-0.05,1,0) to [bend left](0.1,1,0);
                \draw [line width=0.3mm, red] (-.45,-1,-.05) to [bend right](-.3,-1,0);
                \draw[->,thick,red](-2.8,-1,0)--(-2.8,-1,-0.1);

                \draw[->] (-3,0,2.2)--(-2.1,0,1.2);
                \node[red] at (-3,0,2.5) {$\gamma^1([t_1^-,t_1^+])$};

                \draw [line width=0.3mm, green!60!red] plot [smooth, tension=.7] coordinates {(-.46,-1,-.05)(-2,-1,0.8)(-2.5,-1,0)(-2,-1,-0.8)(-0.07,1,0)};
                \draw [line width=0.3mm, green!60!red] (-0.07,1,0) to [bend left](0.1,1,0);
                \draw [line width=0.3mm, green!60!red] (-.46,-1,-.05) to [bend right](-.3,-1,0);
                \draw[->] (-3,0,-1.8)--(-2.1,0,-.5);
                \node[green!60!red] at (-3,0,-2) {$\gamma^2([t_1^-,t_1^+])$};
                 \draw[->,thick,green!60!red](-2.5,-1,0)--(-2.5,-1,-0.1);

                \draw [fill=blue] (-.3,-1,0) circle (2pt);
                \draw [->, thick] (-.6,-1,1)--(-.35,-1,0.2);
                \node[blue] at (-.6,-1,1.3) {\large$\gamma^1(t_1^-)$};
                
                \draw [fill=blue] (0.1,1,0) circle (2pt);
                \draw [->, thick] (-0.4,1,-1)--(0.05,1,-0.1);
                \node[blue] at (-0.4,1,-1.5) {\large$\gamma^1(t_1^+)$};
                
                \draw [fill=blue] (4.2,-1,0.5) circle (2pt); 
                \node[thick,blue] at (4.2,-1,0.9) {\large$x=\gamma^1(0)$};
                \draw[->,thick,blue] (2.5,-2,0.2)--(2.4,-2,0.2); 
                 \draw[->,thick,blue] (5,1,0)--(5.1,1,0); 
                
            \draw [line width=0.3mm, blue] plot [smooth] coordinates {(-.3,-1,0) (1,-1,0.2) (2.5,-2,0.2) (4.2,-1,0.5)  (5,1,0) (0.1,1,0)};

            \draw[rotate=100] (-.5,0,0) to[out=-100,in=130] (-.5,2,-5.5);
            \draw[rotate=100] (.5,0,0) to[out=-100,in=90] (.5,-2,-5.5);
			\draw[rotate=100] (-.5,2,-5.5) to[out=-50,in=-90] (.5,-2,-5.5);

				\draw (-1,0,0) to[bend left] (-2,0,0);
				\draw (-1.2,0,-0.1) to[bend right] (-1.8,0,-0.1);
		
				\draw (1,0,0) to[bend right] (2,0,0);
				\draw (1.2,0,-0.1) to[bend left] (1.8,0,-0.1);

                \draw (3.5,0,0.2) to[bend right] (4.5,0,0.2);
				\draw (3.7,0,.1) to[bend left] (4.3,0,.1);

		\end{tikzpicture}
      \caption{The case when $\gamma^1$ intersects non-trivially both $\Sigma^\circ$ and $\left(M \setminus \Sigma\right)^\circ$.}
      \label{figure: case ii}
      \end{subfigure}%
      \begin{subfigure}{.5\textwidth}
            \centering
            \hspace*{1em}\raisebox{1.5em}{\begin{tikzpicture}[x={(1cm,0cm)},y={(-.2cm,-.2cm)},z={(0cm,1cm)},scale=1]

                \node[] at (-1.8,0,1.8) {\large $\Sigma$};
    			\draw[rotate=90] (0,0,-0.5) to[bend left] (0,0,0.5) node at (0,0,1) {\large$\partial\Sigma$};
				\draw[dashed,rotate=90] (0,0,-0.5) to[bend right] (0,0,0.5);
				\draw[rotate=180] (-.5,0,0) to[out=-90,in=90] (-1.5,0,-2);
				\draw[rotate=180] (-1.5,0,-2) to[out=-90, in=-90] (1.5,0,-2);
				\draw[rotate=180] (.5,0,0) to[out=-90,in=90] (1.5,0,-2);
                \draw (-0.5,0.5,1.5) to[bend right] (.5,0.5,1.5);
				\draw (-0.3,0.5,1.4) to[bend left] (.2,0.5,1.4);

                 \draw [red] (0,0.5,1.5) circle (18pt) node at (0,0.5,2.5) {\large$\gamma_c^1$}; 
                 \draw[->,thick,red] (0.65,0.5,1.5);
                 \draw[->,thick,red] (-0.65,0.5,1.5)--(-0.65,0.5,1.4);
                 \draw [fill=red] (0,0,-.15) circle (1.5pt) node at (0,0,-.45) {\large$x$};
                 \draw [fill=red] (0,0,0.75) circle (1.5pt) node at (0.2,0,0.6) {\large$p$};
                 \draw[red](0,0,-.15)--(0,0,0.75);
		\end{tikzpicture}}
      \caption{The case when $\gamma^1_c$ is in $\Sigma$.}
      \label{figure: case iii}
       \end{subfigure}
       \caption{}
\end{figure}

\textbf{Case (iii).} Assume now that $\gamma^1_c$ is in $\Sigma$. Let $x\in\partial\Sigma$ and $p\in \gamma^1_c$ be a closest point to $x$. Let $\eta_n$ be a curve that starts at $x$, then follows a shortest geodesic $[x,p]$ from $x$ to $p$, then goes $n$ times around $\gamma^1_c$, and then comes back to $x$ along the same geodesic $[p,x]$ (see Figure \ref{figure: case iii}). Denote by $[\eta_n]$ the homotopy class with fixed endpoints being $x$ of $\eta_n$. By \cite[Lemma 2.2]{Guillarmou-Mazzucchelli-18}, $d_{g_1}(x,x,[\eta_n]])\leq n\ell_{g_1}(\gamma^1_c)+2d_{g_1}(x,\gamma^1_c)$. Let $\gamma_{x,x}^1$ be the unique $g_1$-geodesic from $x$ to $x$ in the class $[\eta_n]$. Notice that $\gamma^1_{x,x}$ is a curve in the free homotopy class $nc$ so $\ell_{g_1}(\gamma_{x,x})\geq n\ell_{g_1}(\gamma^1_c)$. Thus, we have $d_{g_1}(x,x,[\eta_n])/n \xrightarrow[]{n \to \infty} \ell_{g_1}(\gamma_c^1)$ and, similarly, $d_{g_2}(x,x,[\eta_n])/n \xrightarrow[]{n \to \infty}\ell_{g_2}(\gamma_c^2)$. Since $d_{g_1}=d_{g_2}$, we obtain that $\ell_{g_1}(\gamma^1_c)=\ell_{g_2}(\gamma^2_c)$.

\end{proof}

We now prove Theorem \ref{theorem:transfer}:

\begin{proof}[Proof of Theorem \ref{theorem:transfer}]
By assumption, there exist extensions $g_1'$ and $g_2'$ to the closed manifold $M$ such that $g_1',g_2'$ are both Anosov, and $g_1'=g_2'$ on $M \setminus \Sigma^\circ$. Fix $\eps_0 > 0$ small enough such that any metric $g$ on $M$ such that $\|g-g_1'\|_{C^2} < \eps_0$ or $\|g-g_2'\|_{C^2} < \eps_0$ is still Anosov (such an $\eps_0 > 0$ exists by Anosov structural stability). By Lemma \ref{lemma:kill-isometries}, for any fixed $\eps > 0$, there exists $f \in C^\infty_0(M \setminus \Sigma)$, vanishing to infinite order on $\partial \Sigma$ such that $\|f\|_{C^2} < \eps$, the metric $e^{2f} g_1'$ has no isometries in $M \setminus \Sigma^\circ$. Moreover, by Lemma \ref{lemma:kill-isometries} again, there exists $\delta > 0$ such that for any $\varphi \in C^4(M \setminus \Sigma^\circ)$ such that $\|\varphi\|_{C^4} < \delta$, the metric $e^{2\varphi} e^{2f} g_1'$ has no isometry in $M \setminus \Sigma$ too. We apply this with $\eps$, $\delta > 0$ small enough so that $\|e^{2\varphi}e^{2f}g_1'-g_1'\|_{C^2(M)} < \eps_0$. Then, we apply Proposition \ref{proposition:technical} with $m \geq k_0$ (where $k_0$ is provided by the injectivity of the marked length spectrum on $M$ in finite regularity) in order to obtain a function $\varphi \in C^m_{0;m}(M\setminus \Sigma^\circ)$ such that:
\begin{enumerate}[label=(\roman*)]
    \item $\|\varphi\|_{C^m} < \delta$ so that $e^{2\varphi} e^{2f} g_1'$ has no global isometries in $M \setminus \Sigma^\circ$;
    \item $e^{2\varphi} e^{2f} g_1'$ satisfies that there are no local isometries between $\Sigma$ and $M \setminus \Sigma^\circ$ (cf. statement of Proposition \ref{proposition:technical});
    \item $\varphi$ is small enough so that $\|e^{2\varphi}e^{2f}g_1'-g_1'\|_{C^2(M)} < \eps_0$ and thus the metric $e^{2\varphi}e^{2f}g_1'$ is Anosov on $M$ (and the same holds for $g_1'$ replaced by $g_2'$).
\end{enumerate}
The metrics $g_1'' := e^{2\varphi} e^{2f} g_1'$ and $g_2'' := e^{2\varphi} e^{2f} g_2'$ are both Anosov and $C^m$-regular on $M$. By Lemma \ref{lemma:same-mls}, they have same marked length spectrum on $M$. Hence, by marked length spectrum rigidity on $M$ (in finite regularity $m$), there exists a $C^{m+1}$-regular diffeomorphism $\phi : M \to M$, isotopic to the identity, such that $g_2'' = \phi^*g_1''$.

We claim that $\phi(\Sigma) = \Sigma$ and $\phi|_{\partial \Sigma} = \mathbf{1}_{\partial \Sigma}$. First, if $\phi(\Sigma) \neq \Sigma$, there exists $y \in \Sigma^\circ$ and $x \in M \setminus \Sigma^\circ$, and open neighborhoods $U_y \subset \Sigma^\circ$ and $U_x \subset M \setminus \Sigma^\circ$ such that $\phi : U_y \to U_x$ is an isometry for $g_1''$ (since $g_2''=g_1''$ on $U_x$). But this contradicts the non-existence of local isometries provided by Proposition \ref{proposition:technical}. Hence $\phi(\Sigma)=\Sigma$ and thus $\phi(M\setminus \Sigma^\circ) = M\setminus \Sigma^\circ$. But then, $\phi$ is an isometry for $g_1''$ on $M\setminus \Sigma^\circ$ so by construction, this forces $\phi$ to be the identity on $M \setminus \Sigma^\circ$. In particular, $\phi|_{\partial \Sigma} = \mathbf{1}_{\partial \Sigma}$. Also, observe that $\phi\colon \Sigma \to \Sigma$ is an isometry (that is, $\phi^*g_1 = g_2$) so by \cite[Theorem 2.1]{Taylor-06}, we get that $\phi$ is smooth if $g_1$ and $g_2$ are smooth. This concludes the proof. 
\end{proof}

\subsection{Mapping class of the isometry}

\label{ssection:topology}

We now further discuss the mapping class of the isometry provided by Theorem \ref{theorem:transfer}. We denote by $\mathrm{Mod}(\Sigma) := \mathrm{Diff}(\Sigma,\partial \Sigma) / \mathrm{Diff}_0(\Sigma,\partial \Sigma)$ the mapping class group of $\Sigma$, that is diffeomorphisms fixing the boundary modulo diffeomorphisms isotopic to the identity via an isotopy fixing the boundary, and by $\mathrm{Mod}(M) := \mathrm{Diff}^+(M)/\mathrm{Diff}_0(M)$ the orientation-preserving diffeomorphisms on $M$ modulo the ones isotopic to the identity.

The embedding $\Sigma \hookrightarrow M$ induces a homomorphism
\begin{equation}
    \label{equation:eta}
\eta \colon \mathrm{Mod}(\Sigma) \to \mathrm{Mod}(M)
\end{equation}
defined as follows: given a representative $\phi \in \mathrm{Diff}(\Sigma,\partial \Sigma)$ of $[\phi] \in \mathrm{Mod}(\Sigma)$, one can isotope it to a diffeomorphism $\psi \colon \Sigma \to \Sigma$ via an isotopy preserving the boundary such that $\psi=\mathbf{1}$ on a collar neighborhood of $\partial \Sigma$; extending $\psi$ by the identity on $M \setminus \Sigma$ yields a diffeomorphism $\psi'\colon M \to M$ whose class $[\psi']$ in $\mathrm{Mod}(M)$ defines $\eta([\phi])$.

The following holds:

\begin{lemma}
    \label{lemma:diff}
    If the induced homomorphism $\eta\colon \mathrm{Mod}(\Sigma)\rightarrow\mathrm{Mod}(M)$ is injective, then the isometry $\phi\colon \Sigma \to \Sigma$ in Theorem \ref{theorem:transfer} is in $\mathrm{Diff}_0(\Sigma,\partial \Sigma)$.
\end{lemma}

\begin{proof}
    The proof is straightforward since the isometry on $\Sigma$ in Theorem \ref{theorem:main} is the restriction to $\Sigma$ of an isometry on $M$ isotopic to the identity (hence in the kernel of $\eta$).
\end{proof}

We emphasize that in the boundary rigidity problem, even on simple manifolds, the isometry $\phi\colon \Sigma \to \Sigma$ is \emph{not} necessarily an element of $\mathrm{Diff}_0(\Sigma,\partial \Sigma)$. Indeed, it is well-known that there exists on $\Sigma := D^6$, the $6$-dimensional closed ball in $\R^6$, a diffeomorphism $\phi\colon \Sigma \to \Sigma$ fixing the boundary, but not isotopic to the identity via an isotopy preserving the boundary. Hence, taking any simple metric $g$ on $\Sigma$, it is straightforward to check that $d_g = d_{\phi^*g}$ but the two metrics $g$ and $\phi^*g$ do not differ by an element of $\mathrm{Diff}_0(\Sigma,\partial \Sigma)$. However, in dimension $2$ and $3$, it is known that $\mathrm{Mod}(D^2)=\mathrm{Mod}(D^3)=\left\{0\right\}$, see \cite{Hatcher-83}.

\subsection{Proof of the main theorem}

Using Theorem \ref{theorem:transfer}, the proof of Theorem \ref{theorem:main} now boils down to proving that metrics on open surfaces of Anosov type with same marked boundary distance function are consistently extendable.

\begin{proposition}
\label{proposition:extendability}
Let $\Sigma$ be a smooth compact connected oriented surface with boundary. Let $g_1$ and $g_2$ be two metrics of Anosov type on $\Sigma$ and further assume that $g_1$ and $g_2$ have same marked boundary distance function, i.e. $d_{g_1} = d_{g_2}$. Then there exists a diffeomorphism $\phi \in \mathrm{Diffeo}_0(\Sigma,\partial \Sigma)$ such that $\phi^*g_1$ and $g_2$ are consistently extendable to a closed surface $M$ with genus $\geq 2$ and the induced homomorphism $\eta : \mathrm{Mod}(\Sigma) \to \mathrm{Mod}(M)$ in \eqref{equation:eta} is injective.
\end{proposition}

\begin{proof}
By Lemma \ref{lemma:jets}, we can find a first diffeomorphism $\phi \in \mathrm{Diffeo}_0(\Sigma,\partial \Sigma)$ such that $\phi^*g_1$ and $g_2$ have same $C^\infty$-jets on $\partial \Sigma$. By \cite[Theorem A]{Chen-Erchenko-Gogolev-20}, there exists a closed oriented Anosov surface $(M,g_1')$ (hence of genus $\geq 2$) and an isometric embedding $\iota\colon (\Sigma, \phi^*g_1) \to (M,g_1')$. We identify $\Sigma$ with its image $\iota(\Sigma)$. The metric $g_2'$ defined by $g_2' = g_2$ on $\Sigma$ and $g_2' = g_1'$ on $M \setminus \Sigma$ is smooth since $\phi^*g_1$ and $g_2$ have same $C^\infty$-jets on $\partial \Sigma$. However, it is not clear that the second metric $g_2'$ is Anosov on $M$. Our aim is to review the construction of \cite{Chen-Erchenko-Gogolev-20} to show that one can actually guarantee that both metrics $g_1',g_2'$ are Anosov on $M$.

The construction of \cite{Chen-Erchenko-Gogolev-20} actually produces a family of such isometric embeddings $\iota \colon (\Sigma, \phi^*g_1) \to (M,g_1')$ depending on various parameters $(\delta_0,\eps,\ell,\delta,R)$ listed below. We claim that among this family there exists one $\iota$ such that $g_2'$ is also an Anosov metric on $M$. Indeed, the isometric embedding in \cite[Theorem A]{Chen-Erchenko-Gogolev-20} is constructed in several steps which we now review (see also \cite[Section 1.1]{Chen-Erchenko-Gogolev-20} for an outline of the construction). \\

(i) In the first step, we just construct a slightly larger extension $(\Sigma^{\delta_0},g^{\delta_0}_1)$ so that the original properties of metric $g_1$ are preserved. The reason for that is that the extension of $g_1$ to the closed manifold $M$ is firstly only $C^{1,1}$. Thus, we will need to apply a smoothing procedure but we do not want to perturb the metrics on $\Sigma$. By \cite[Lemma 2.3]{Guillarmou-17-2}, for any sufficiently small $\delta_0>0$, there exists a compact surface $(\Sigma^{\delta_0},g^{\delta_0}_1)$ with boundary and an isometric embedding $\iota_1\colon (\Sigma,g_1)\rightarrow (\Sigma^{\delta_0},g^{\delta_0}_1)$ where $(\Sigma^{\delta_0},g^{\delta_0}_1)$ has strictly convex boundary equidistant to $\partial \Sigma$ at distance $\delta_0$ with respect to the metric $g_1^{\delta_0}$, has the same hyperbolic trapped set as $(\Sigma,g_1)$, and no conjugate points. Moreover, all equidistant curves to $\partial \Sigma$ in $\Sigma^{\delta_0}\setminus \Sigma$ are strictly convex. We identify $\Sigma$ with its image $\iota_1(\Sigma)$. The parameter $\delta_0$ is chosen small enough so that the curvature of all equidistant curves to $\partial\Sigma$ in $\Sigma^{\delta_0}\setminus\Sigma$ is at least $\kappa_{\min}(\partial\Sigma)/2$ where $\kappa_{\min}(\partial\Sigma)$ is the minimal curvature of $\partial\Sigma$. Notice that this embedding depends  only on $C^\infty$-jets of metrics on $\partial\Sigma$. \\

(ii) Let $\partial\Sigma^{\delta_0}=\sqcup_{j=1}^m S_j$ where each $S_j$ is diffeomorphic to a circle. We parametrize each boundary component so that the tangent vector and the outward normal vector form an oriented basis. 
 In the second step, for each $S_j$,  we find a metric on a collar that connects the metric on this boundary to a constant curvature metric. The resulting metric is $C^\infty$ everywhere except on $\partial\Sigma^{\delta_0}$ and an equidistant curve to it. Then, we smoothen the metric so that the metric on $\Sigma$ is untouched.

By \cite[Proposition 7.1]{Chen-Erchenko-Gogolev-20}, for any $\eps<\delta_0$, there exists a compact surface $(\Sigma^*, g^*)$ with boundary and an isometric embedding $\iota_2\colon (\Sigma^{\delta_0}, g_1^{\delta_0})\rightarrow (\Sigma^*,g^*)$ where $(\Sigma^*\setminus \Sigma^\circ,g^*)$ is isometric to $[-\delta_0,4]\times \sqcup_{j=1}^mS_j$ with metric $\tilde{g}_j=dt^2+(\tilde{g_j})_t$ on $[-\delta_0,4]\times S_j$ where $(\tilde{g_j})_t$ is the metric on $(S_j)_t:=\{t\}\times S_j$ and $(\tilde{g_j})_t$ is the smoothing of the $C^{1,1}$ Riemannian metric $(g_j)_t$ on $(S_j)_t$ defined in \eqref{equation:metric} below (see Figure~\ref{figure: CEG} and \cite[Proposition 7.2]{Chen-Erchenko-Gogolev-20}). Note that for any $\theta\in S_j$, we have that $t \mapsto \gamma(t):=(t,\theta)$ (for $t\in[-\delta_0,4]$) is a geodesic on $[-\delta_0,4]\times S_j$ for the metric $g^*$. Also, $\{0\}\times\sqcup_{j=1}^mS_j$ corresponds to $\partial\Sigma^{\delta_0}$, and $\{-\delta_0\}\times\sqcup_{j=1}^mS_j$ corresponds to $\partial \Sigma$.

\begin{figure}[H]
    \tiny
      \centering
			\begin{tikzpicture}[scale=0.65]
            \draw[thick,blue] (0,0) arc (-90:90:0.3 and 1.5);
			\draw [dashed,blue] (0,0) arc (270:90:0.3 and 1.5);
            \draw[thick,red] (-7,-0.85) arc (-90:90:0.5 and 2.35);
			\draw [dashed,red] (-7,-0.85) arc (270:90:0.4 and 2.35);
            \node [red] at (-6.9,1.8) {$\partial\Sigma^{\delta_0}$};

           \draw[thick] (-6,-0.75) arc (-90:90:0.5 and 2.25);
			\draw [dashed] (-6,-0.75) arc (270:90:0.4 and 2.25);

            \draw[thick] (-4,-0.5) arc (-90:90:0.5 and 2);
			\draw [dashed] (-4,-0.5) arc (270:90:0.4 and 2);

            \draw[thick] (-2,-0.25) arc (-90:90:0.5 and 1.75);
			\draw [dashed] (-2,-0.25) arc (270:90:0.4 and 1.75);    
			\draw [thick,purple] (-8,1.5) ellipse (0.5 and 2.5);
            \node[purple] at (-8, 2.5) {$\partial\Sigma$};
			\draw [thick] plot [ smooth, tension = 0.3] coordinates {(-8, -1) (0, 0)}; 
			\draw [thick] plot [smooth, tension = 0.3] coordinates {(-8, 4) (0, 3)};

             \draw [thick, decorate,
    decoration = {calligraphic brace}] (-8,4.2) --  (-7,4.2);
   \node[] at (-7.5,4.6) {\textbf{(A)}};

            \draw [thick, decorate,
    decoration = {calligraphic brace, mirror}] (-7,-1) --  (-6,-1);
            \node[] at (-6.5,-1.4) {\textbf{(B)}};

            \draw [thick, decorate,
    decoration = {calligraphic brace}] (-6,4) --  (-4,4);
        \node[] at (-5,4.4) {\textbf{(C)}};

            \draw [thick, decorate,
    decoration = {calligraphic brace,mirror}] (-4,-0.7) --  (-2,-0.7);
            \node[] at (-3,-1.1) {\textbf{(D)}};

       \draw [thick, decorate,
    decoration = {calligraphic brace}] (-2,3.5) --  (0,3.5);
   \node[] at (-1,3.9) {\textbf{(E)}};
            \end{tikzpicture}
      \caption{The $C^{1,1}$ metric on $[-\delta_0,4]\times\partial\Sigma^\delta_0$ from \cite{Chen-Erchenko-Gogolev-20}. \textbf{(A)} Width $\delta_0$. \textbf{(B)} Width $\eps$; deformation to negative curvature with controlled upper bound on it. \textbf{(C)} Width $1$; negative curvature. \textbf{(D)} Width $1+\eps$; ``rounding'' the metric, negative curvature. \textbf{(E)} Constant negative curvature.}
      \label{figure: CEG}
\end{figure}
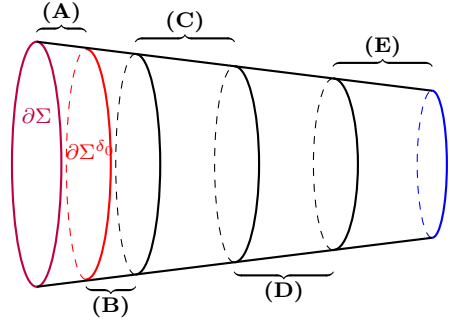

We provide the expression of the metric on the attached collar to show that it only depends on the $C^\infty$ jet of $g_1$ on $\partial\Sigma$. Let $(t,\theta)\in [-\delta_0,4]\times S_j$. Define $\pi_s: \mathbb R\times S_j\to \mathbb R\times S_j$ by $\pi_s(t,\theta):=(t+s,\theta)$ for $\theta\in S_j$. Let $h_j(X,Y) :=2\kappa_{(S_j)_0}(0,\theta)|d\pi_{-t}X|_{g_1^{\delta_0}}|d\pi_{-t}Y|_{g_1^{\delta_0}}$ where $\kappa_{(S_j)_0}$ is the curvature of $(S_j)_0$ and $X,Y\in T_{(t,\theta)}(S_j)_t$. Notice that from the strict convexity condition, we have $\kappa_{(S_j)_0}>0$. Choose a non-increasing $C^\infty$ function $\rho\colon \mathbb R\rightarrow [0,1]$ such that $\rho\equiv 1$ on $(-\infty,0]$ and $\rho\equiv 0$ on $[1,\infty)$. Define a function $f_\ell\colon\mathbb R\rightarrow\mathbb R$ by $f_\ell(t)=\frac{e^{\ell t}-1}{\ell}$ for $t\in\mathbb R$. Let $ds^2$ be the standard metric on $\mathbb R/(2\pi)\mathbb Z$. Then,
\begin{equation}
\label{equation:metric}
(g_j)_t = \left\{\begin{aligned} &g^{\delta_0}_t, \qquad t\in[-\delta_0,0),\\
&\rho(t-\eps)g^{\delta_0}_0+f_{\ell}(t)h_j, \qquad t\in[0,1+\eps],\\
&f_{\ell}(t)\left(\rho(t-1-\eps)h_j+(1-\rho(t-1-\eps)ds^2)\right), \qquad t\in[1+\eps,2+2\eps],\\
&\left(\frac{1}{\kappa}\sinh[\kappa(t+\tilde r)]\right)^2ds^2, \qquad t\in[2+2\eps,4]
\end{aligned}\right.
\end{equation}
where $\kappa$ and $\tilde r>-2-2\eps$ are determined by $\ell$, and $-\kappa^2\rightarrow -\infty$ as $\ell\rightarrow\infty$ by \cite[Lemma C.1]{Chen-Erchenko-Gogolev-20}. The smoothing of $g_j$ is done in the $\delta$-neighborhood (for $\delta\in(0,\frac{\eps}{2})$) of $t=0$ and $t=2+2\eps$. \\

(iii) The final step is to glue the manifold $(\Sigma^*,g^*)$ with boundary into a closed manifold of the same dimension so that the boundary components of $\partial\Sigma^*$ are far enough and a geodesic that leaves $\Sigma^*$ does not return back to it too soon. 

Fix $R>0$. By the previous item, $g^*$ is the hyperbolic metric of constant curvature $-\kappa^2$ on the neighborhood of $\partial\Sigma^*$ and correspond to a metric on the annulus. We can remove $m$ disks from a hyperbolic surface of curvature $-\kappa^2$  at a distance at least $R$ from each other (see \cite[Lemma 8.3]{Chen-Erchenko-Gogolev-20}) and attach $(\Sigma^*,g^*)$. As a result, we obtain an isometric embedding $\iota\colon (\Sigma^*,g^*)\rightarrow (M,g_1')$ where $(M,g_1')$ is a closed oriented surface of genus $\geq 2$. It can be further constructed so that no connected component of $M\setminus \Sigma^*$ is diffeomorphic to an open annulus or an open disk. By \cite[Theorem 3.18]{Farb-Margalit-12}, this guarantees that the induced homomorphism $\eta : \mathrm{Mod}(\Sigma) \to \mathrm{Mod}(M)$ is injective. \\

By choosing $\delta_0$ and $\eps$ sufficiently small, $\ell$ sufficiently large, $\delta$ sufficiently small (depending on the choice of $\eps$ and $\ell$), and $R$ sufficiently large depending on the metrics $g_1$ and $g_2$, we can guarantee that $g_1'$ and $g_2'$ are both Anosov metrics on $M$, see \cite[Section 8]{Chen-Erchenko-Gogolev-20} for further details.

\end{proof}

We can now conclude the proof of Theorem \ref{theorem:main}:

\begin{proof}[Proof of Theorem \ref{theorem:main}]
Straightforward consequence of Proposition \ref{proposition:extendability}, Theorem \ref{theorem:transfer} and the marked length spectrum rigidity of closed Anosov surfaces with $C^4$-regularity proved in \cite{Guillarmou-Lefeuvre-Paternain-23}. The isometry $\phi$ is an element in $\mathrm{Diff}_0(\Sigma,\partial \Sigma)$ by injectivity of $\eta\colon \mathrm{Mod}(\Sigma) \to \mathrm{Mod}(M)$ (Proposition \ref{proposition:extendability} and Lemma \ref{lemma:diff}).
\end{proof}

\bibliographystyle{alpha}
\bibliography{Biblio}

\end{document}